\documentclass[11pt]{article}
\usepackage[margin=1.25in]{geometry}
\usepackage{amssymb,amsfonts,amsmath,amsthm,amscd,dsfont,mathrsfs,bbold,pifont}
\usepackage{blkarray}
\usepackage{graphicx,float,psfrag,epsfig,color}
\usepackage{comment}
\usepackage{subcaption}

\usepackage{algorithm}
\usepackage[noend]{algpseudocode}

\makeatletter
\def\BState{\State\hskip-\ALG@thistlm}
\makeatother

	\usepackage{microtype}
	\usepackage[pdftex,pagebackref=true,colorlinks]{hyperref}
	\hypersetup{linkcolor=[rgb]{.7,0,0}}
	\hypersetup{citecolor=[rgb]{0,.7,0}}
	\hypersetup{urlcolor=[rgb]{.7,0,.7}}

\newcommand\independent{\protect\mathpalette{\protect\independent}{\perp}} 
\def\independent#1#2{\mathrel{\rlap{$#1#2$}\mkern2mu{#1#2}}}

\newcommand{\gss}{\mathbb{G}_1}
\newcommand{\gs}{\mathbb{G}_2}
\newcommand{\bin}{\mathrm{Bin}}
\newcommand{\dd}{D_+}
\newcommand{\argmax}{\mathrm{argmax}}

\newcommand{\mR}{\mathbb{R}} 

\newcommand{\mZ}{\mathbb{Z}}

\newcommand{\pp}{\mathbb{P}}

\newcommand{\e}{\varepsilon}

\DeclareMathOperator{\diag}{diag}

\newcommand{\X}{\mathcal{X}}

\newtheorem{theorem}{Theorem}
\newtheorem{lemma}{Lemma}
\newtheorem{corollary}{Corollary}
\newtheorem{definition}{Definition}
\newtheorem{remark}{Remark}
\newtheorem{example}{Example}

\begin{document}

\title{Community detection in general stochastic block models: \\ fundamental limits and efficient recovery algorithms}
\author{Emmanuel Abbe and Colin Sandon}

\author{Emmanuel Abbe\thanks{Program in Applied and Computational Mathematics, and EE department, Princeton University, Princeton, USA, \texttt{eabbe@princeton.edu}. This research was partially supported by the 2014 Bell Labs Prize.} \and 
Colin Sandon\thanks{Department of Mathematics, Princeton University, USA,
\texttt{sandon@princeton.edu}.
}}

\date{}
\maketitle

\begin{abstract}
New phase transition phenomena have recently been discovered for the stochastic block model, for the special case of two non-overlapping symmetric communities. This gives raise in particular to new algorithmic challenges driven by the thresholds. This paper investigates whether a general phenomenon takes place for multiple communities, without imposing symmetry.

In the general stochastic block model $\text{SBM}(n,p,Q)$, $n$ vertices are split into $k$ communities of relative size $\{p_i\}_{i \in [k]}$, and  vertices in community $i$ and $j$ connect independently with probability $\{Q_{i,j}\}_{i,j \in [k]}$. This paper investigates the partial and exact recovery of  communities in the general SBM (in the constant and logarithmic degree regimes), and uses the generality of the results to tackle overlapping communities.

The contributions of the paper are: (i) an explicit characterization of the recovery threshold in the general SBM in terms of a new divergence function $D_+$, which generalizes the Hellinger and Chernoff divergences, and which provides an operational meaning to a divergence function analog to the KL-divergence in the channel coding theorem, (ii) the development of an algorithm that recovers the communities all the way down to the optimal threshold and runs in quasi-linear time, showing that exact recovery has no information-theoretic to computational gap for multiple communities, in contrast to the conjectures made for detection with more than 4 communities; note that the algorithm is optimal both in terms of achieving the threshold and in having quasi-linear complexity, (iii) the development of an efficient algorithm that detects communities in the constant degree regime with an explicit accuracy bound that can be made arbitrarily close to 1 when a prescribed signal-to-noise ratio (defined in term of the spectrum of $\diag(p)Q$) tends to infinity. 

%
\end{abstract}

\thispagestyle{empty}
\newpage

\tableofcontents

\thispagestyle{empty}
\newpage

\pagenumbering{arabic}


\section{Introduction}



Detecting communities (or clusters) in graphs is a fundamental problem in computer science and machine learning. This applies to a large variety of complex networks in social sciences and biology, as well as to data sets engineered as networks via similarly graphs, where one often attempts to get a first impression on the data by trying to identify groups with similar behavior. 
In particular, finding communities allows one to find like-minded people in social networks \cite{newman-girvan,social1}, to improve recommendation systems \cite{amazon,xu-rec}, to segment or classify images \cite{image1,image2}, to detect protein complexes \cite{ppi2,marcotte}, to find genetically related sub-populations \cite{genetics,gene-survey}, or to discover new tumor subclasses \cite{tumor}. 

While a large variety of community detection algorithms have been deployed in the past decades,  understanding the fundamental limits of community detection and establishing rigorous benchmarks for algorithms remains a major challenge. Significant progress has recently been made for the stochastic block model, but mainly for the special case of two non-overlapping communities. The goal of this paper is to establish the fundamental limits of recovering communities in general stochastic block models, with multiple (possibly overlapping) communities. We first provide some motivations behind these questions.  

Probabilistic network models can be used to model real networks \cite{newman-book}, to study the average-case complexity of NP-hard problems on graphs (such as min-bisection or max-cut \cite{dyer,bui,condon,bollobas-cut}), or 
to set benchmarks for clustering algorithms with well defined ground truth. In particular, the latter holds irrespective of how exactly the model fits the data sets, and is a crucial aspect in community detection as a vast majority of algorithms are based on heuristics and no ground truth is typically available in applications. This is in particular a well known challenge for Big Data problems where one cannot manually determine the quality of the clusters \cite{asa}. 

Evaluating the performance of algorithms on models is, however, non-trivial. In some regimes, most reasonable algorithms may succeed, while in others, algorithms may be doomed to fail due to computational barriers. Thus, an important question is to characterize the regimes where the clustering tasks can be solved efficiently or information-theoretically. In particular, models may benefit from asymptotic phase transition phenomena, which, in addition to being mathematical interesting, allow location of the bottleneck regimes to benchmark algorithms. Such phenomena are commonly used in coding theory (with the channel capacity \cite{shannon48}), or in constraint satisfaction problems (with the SAT thresholds, see  \cite{AchlioetalNature} and references therein). 

Recently, similar phenomena have been identified for the stochastic block model (SBM), one of the most popular network models exhibiting community structures \cite{holland,sbm1,sbm3,sbm4,bickel,newman2}. The model\footnote{See Section \ref{related_lit} for further references.} was first proposed in the 80s \cite{holland} and  received significant attention in the mathematics and computer science literature \cite{bui,dyer,boppana,jerrum,condon,carson}, as well as in the statistics and machine learning literature \cite{snij,bickel,rohe,choi}. The SBM puts a distribution on $n$-vertices graphs with a hidden (or planted) partition of the nodes into $k$ communities. Denoting by $p_i$, $i \in [k]$, the relative size of each community, and assuming that a pair of nodes in communities $i$ and $j$ connects independently with probability $Q_{i,j}$, the SBM can be defined by the triplet $(n,p,Q)$, where $p$ is a probability vector of dimension $k$ and $Q$ a $k\times k$ symmetric matrix with entries in $[0,1]$. 

The SBM recently came back at the center of the attention at both the practical level, due to extensions allowing overlapping communities \cite{mixed-core} that have proved to fit well real data sets in massive networks \cite{prem}, and at the theoretical level due to new phase transition phenomena discovered for the two-community case \cite{coja-sbm,decelle,massoulie-STOC,Mossel_SBM2,abh,mossel-consist}. To discuss these phenomena, we need to first introduce the figure of merits (formal definitions are in Section \ref{res-sec}):
\begin{itemize}
\item  {\bf Weak recovery} (also called detection). This only requires the algorithm to output a partition of the nodes which is  positively correlated with the true partition (whp\footnote{whp means with high probablity, i.e., with probability $1-o_n(1)$ when the number of nodes in the graph diverges.}). Note that weak recovery is relevant in the fully symmetric case where all nodes have identical average degree,\footnote{At least for the case for communities having linear size. One may otherwise define stronger notions of weak recovery that apply to non-symmetric cases.} since otherwise weak recovery can be trivially solved. If the model is perfectly symmetric, like the SBM with two equally-sized clusters having the same connectivity parameters, then weak recovery is non-trivial. Full symmetry may not be representative of reality, but it sets analytical and algorithmic challenges. The weak-recovery threshold for two symmetric communities was achieved efficiently in \cite{massoulie-STOC,Mossel_SBM2}, settling a conjecture established in \cite{decelle}. The case with more than two communities remains open. 
\item {\bf Partial recovery.} One may ask for the finer question of {\it how much} can be recovered about the communities. For a given set of parameters of the block model, finding the proportion of nodes (as a function of $p$ and $Q$) that can be correctly recovered (whp) is an open problem. Obtaining a closed form formula for this question is unlikely, even in the symmetric case with two communities. Partial results were obtained in \cite{mossel2} for two-symmetric communities, but the general problem remains open even for determining scaling laws. 
One may also consider the special case of partial recovery where only an $o(n)$ fraction of nodes is allowed to be mis-classified (whp), called almost exact recovery or weak consistency, but no sharp phase transition is to be expected for this requirement.
\item {\bf Exact recovery} (also called recovery or strong consistency.) Finally, one may ask for the regimes for which an algorithm can recover the entire clusters (whp). This is non-trivial for both symmetric and asymmetric parameters. One can also study ``partial-exact-recovery,'' namely, which communities can be exactly recovered. While exact recovery has been the main focus in the literature for the past decades (see table in Section \ref{related_lit}), the phase transition for exact recovery was only obtained last year for the case of two symmetric communities \cite{abh,mossel-consist}. The case with more than two communities remains open. 
\end{itemize}
This paper addresses items 2 and 3 for the general stochastic block model. 
Note that the above questions naturally require studying different regimes for the parameters. Weak recovery requires the edge probabilities to be $\Omega(1/n)$, in order to have many vertices in all but one community to be non-isolated  (i.e., a giant component in the symmetric case), and recovery requires the edge probabilities to be $\Omega(\ln(n)/n)$, in order to have all vertices in all but one community to be non-isolated (i.e., a connected graph in the symmetric case). The difficulty is to understand how much more is needed in order to weakly or exactly recover the communities. In particular, giants and connectivity have phase transition, and similar phenomena may be expected for weak and exact recovery. 

Note that these regimes are not only rich mathematically, but are also relevant for applications, as a vast collection of real networks ranging from social (LinkedIn, MSN), collaborative (movies, arXiv), or biological (yeast) networks and more were shown to be sparse \cite{sparse-network1,sparse-network2}. Note however that the average degree is typically not small in real networks, and it seems hard to distinguish between a large constant or a slowly growing function. Both regimes are of interest to us.  


Finally, there is an important distinction to be made between the information-theoretic thresholds, which do not put constraints on the algorithm's complexity, and the computational thresholds, which require polynomial-time algorithms. In the case of two symmetric communities, the information-theoretic and computational thresholds were proved to be the same for weak recovery \cite{massoulie-STOC,Mossel_SBM2} and exact recovery \cite{abh,mossel-consist}. A gap is conjectured to take place for weak recovery for more than 4 communities \cite{mossel-sbm}. No conjectures were made for exact recovery for multiple communities.

This paper focuses on partial and exact recovery (items 2 and 3) for the general stochastic block model with linear size communities, and uses the generality of the results to address overlapping communities (see Section \ref{overlap}).
Recall that for the case of two communities, if 
\begin{align*}
q_{in}&=a \ln(n)/n,\\
q_{out}&=b \ln(n)/n,
\end{align*}
are respectively the intra- and extra-cluster probabilities, with $a>b>0$, then exact recovery is possible if and only if
\begin{align}
\sqrt{a} - \sqrt{b} \geq \sqrt{2}, \label{2hell}
\end{align}
and this is efficiently achievable. However, there is currently no general insight regarding equation \eqref{2hell}, as it emerges from estimating a tail event for Binomial random variable specific to the case of two-symmetric communities. Moreover, no results are known to prove partial recovery bounds for more than two communities (recent progress where made in \cite{new-vu}). This represents a limitation of the current techniques, and an impediment to progress towards more realistic network models that may have overlapping communities, and for which analytical results are currently unknown.\footnote{Different models than the SBM allowing for overlapping communities have been studied for example in \cite{arora-overlap}.} We next present our effort towards such a general treatment.




\section{Results}\label{res-sec}
The main advances of this paper are:
\begin{itemize}
\item[(i)] an ({\tt Sphere-comparison}) algorithm that detect communities in the constant-degree general SBM with an explicit accuracy guarantee, such that when a prescribed signal-to-noise ratio --- defined in terms of the ratio $|\lambda_{\mathrm{min}}|^2/\lambda_{\mathrm{max}}$ where $\lambda_{\mathrm{min}}$ and $\lambda_{\mathrm{max}}$ are respectively the smallest\footnote{The smallest eigenvalue $\diag(p)Q$ is the one with least magnitude.} and largest eigenvalue of $\diag(p)Q$ --- tends to infinity, the accuracy tends to 1 and the algorithm complexity becomes quasi-linear, i.e., $o(n^{1+ \e})$, for all $\e>0$,
\item[(ii)] an explicit characterization of the recovery threshold in the general SBM in terms of a divergence function $D_+$, which provides a new operational meaning to a divergence analog to the KL-divergence in the channel coding theorem (see Section \ref{it-inter}), and which allows determining which communities can be recovered by solving a packing problem in the appropriate embedding, 
\item[(iii)] a quasi-linear time algorithm ({\tt Degree-profiling}) that solves exact recovery whenever it is information-theoretically solvable\footnote{Assuming that the entries of $Q_{ij}$ are non-zero --- see Remark \ref{qzero} for zero entries.}, showing in particular that there is no information-theoretic to computational gap for exact recovery with multiple communities, in contrast to the conjectures made for weak recovery. Note that the algorithm replicates statistically the performance of maximum-likelihood (which is NP-hard in the worst-case) with an optimal (i.e., quasi-linear) complexity. In particular, it improves significantly on the SDPs developed for two communities (see Section \ref{related_lit}) both in terms of generality and complexity.
\end{itemize}

\subsection{Definitions and terminologies}\label{def-term}
The general stochastic block model, $\text{SBM}(n,p,Q)$, is a random graph ensemble defined as follows:
\begin{itemize}
\item $n$ is the number of vertices in the graph, $V=[n]$ denotes the vertex set.
\item Each vertex $v \in V$ is assigned independently a hidden (or planted) label $\sigma_v$ in $[k]$ under a probability distribution $p=(p_1,\dots,p_k)$ on $[k]$. That is, $\pp\{\sigma_v=i\}=p_i$, $i \in[k]$. We also define $P=\diag(p)$.
\item Each (unordered) pair of nodes $(u,v) \in V\times V$ is connected independently with probability $Q_{\sigma_u,\sigma_v}$, where $Q_{\sigma_u,\sigma_v}$ is specified by a symmetric $k \times k$ matrix $Q$ with entries in $[0,1]$.
\end{itemize}
The above gives a distribution on $n$-vertices graphs. 
Note that $G\sim SBM(n,p,Q)$ denotes a random graph drawn under this model, without the hidden (or planted) clusters (i.e., the labels $\sigma_v$ ) revealed. The goal is to recover these labels by observing only the graph. 

This paper focuses on $p$ independent of $n$ (the communities have linear size), $Q$ dependent on $n$ such that the average node degrees are either constant or logarithmically growing and $k$ fixed. These assumptions on $p$ and $k$ could be relaxed, for example to slowly growing $k$, but we leave this for future work. As discussed in the introduction, the above regimes for $Q$ are both motivated by applications and by the fact that interesting mathematical phenomena take place in these regimes. For convenience, we attribute specific notations for the model in these regimes:
\begin{definition}
For a symmetric matrix $Q \in \mR_+^{k \times k}$, 
\begin{itemize}
\item $\gss(n,p,Q)$ denotes $\text{SBM}(n,p,Q/n)$,
\item $\gs(n,p,Q)$ denotes $\text{SBM}(n,p,\ln(n)Q/n)$.
\end{itemize}
\end{definition}

We now discuss the recovery requirements. 
\begin{definition} (Partial recovery.) 
An algorithm recovers or detects communities in $\text{SBM}(n,p,Q)$ with an accuracy of $\alpha \in [0,1]$, if it outputs a labelling of the nodes $\{\sigma'(v), v \in V\}$, which agrees with the true labelling $\sigma$ on a fraction $\alpha$ of the nodes with probability $1-o_n(1)$. The agreement is maximized over relabellings of the communities.
\end{definition}

\begin{definition} (Exact recovery.) 
Exact recovery is solvable in $\text{SBM}(n,p,Q)$ for a community partition $[k] = \sqcup_{s=1}^t A_s$, where $A_s$ is a subset of $[k]$, if there exists an algorithm that takes $G \sim SBM(n,p,Q)$ and assigns to each node in $G$ an element of $\{A_1,\dots,A_t\}$ that contains its true community\footnote{This is again up to relabellings of the communities.} with probability $1-o_n(1)$.  Exact recovery is solvable in $\text{SBM}(n,p,Q)$ if it is solvable for the partition of $[k]$ into $k$ singletons, i.e., all communities can be recovered. The problem is solvable information-theoretically if there exists an algorithm that solves it, and efficiently if the algorithm runs in polynomial-time in $n$. 
\end{definition}
Note that exact recovery for the partition $[k]=\{i\} \sqcup ([k] \setminus \{i\})$ is equivalent to extracting community $i$. In general, recovering a partition $[k] = \sqcup_{s=1}^t A_s$ is equivalent to merging the communities that are in a common subset $A_s$ and recovering the merged communities. 
Note also that exact recovery in $\text{SBM}(n,p,Q)$ requires the graph not to have vertices of degree $0$ in multiple communities (with high probability). In the symmetric case, this amounts to asking for connectivity. 
Therefore, for exact recovery, we will focus below on $Q$ scaling like $\frac{\ln(n)}{n}Q$ where $Q$ is a fixed matrix, i.e., on the $\gs(n,p,Q)$ model.

\subsection{Main results}\label{main-res}
We next present our main results and algorithms for partial and exact recovery in the general SBM. We present slightly simplified versions in this section, and provide full statements in Sections \ref{partial-sec} and \ref{exact-sec}. \\

\noindent
{\bf The CH-embedding and exact recovery.} We explain first how to identify the communities that can be extracted from a graph drawn under $\gs(n,p,Q)$.  Define first the community profile of community $i \in [k]$ by the vector 
\begin{align}
\theta_i:=(PQ)_i \in \mR_+^k,
\end{align} 
i.e., the $i$-th column of the matrix $\diag(p)Q$. Note that $\| \theta_i \|_1 \log(n)$ gives the average degree of a node in community $i$. Two communities having the same community profile cannot be distinguished, in that the random graph distribution is invariant under any permutation of the nodes in these communities. Intuitively, one would expect that the further ``apart'' the community profiles are, the easier it should be to distinguish the communities. The challenge is to quantify what ``apart'' means, and whether there exists a proper distance notion to rely on. We found that the following function gives the appropriate notion, 
\begin{align} 
\dd:\, \mR_+^k \times \mR_+^k &\,\, \to \,\, \mR_+ \notag\\
\phantom{\dd:} (\theta_i, \theta_j) &\,\, \mapsto \,\, \dd(\theta_i, \theta_j) =\max_{t \in [0,1]} \sum_{x \in [k]} \left( t\theta_i(x) + (1-t)\theta_j(x)- \theta_i(x)^t \theta_j(x)^{1-t} \right). \label{h-div1}
\end{align}
For a fixed $t$, the above is a so-called $f$-divergence (obtained for $f(x)=1-t+tx-x^t$), a family of divergences generalizing the KL-divergence (relative entropy) defined in \cite{csiszar-f,morimoto,ali} and used in information theory and statistics. As explained in Section \ref{it-inter}, $\dd$ can be viewed as a generalization of the Hellinger divergence (obtained for $t=1/2$) and the Chernoff divergence. We therefore call $\dd$ the Chernoff-Hellinger (CH) divergence.  Note that for the case of two symmetric communities, $\dd(\theta_1,\theta_2)=\frac{1}{2}(\sqrt{a}-\sqrt{b})^2$, recovering the result in \cite{abh,mossel-consist}.

To determine which communities can be recovered, partition the community profiles into the largest collection of disjoint subsets such that the CH-divergence among these subsets is at least 1 (where the $H$-divergence between two subsets of profiles is the minimum of the $H$-divergence between any two profiles in each subset). We refer to this as the {\it finest partition} of the communities. Figure \ref{finest-partition} illustrates this partition. The theorem below shows that this is indeed the most granular partition that can be recovered about the communities, in particular, it characterizes the information-theoretic and computational threshold for exact recovery. 

\begin{figure}[h]
\centering
  \includegraphics[width=.5\linewidth]{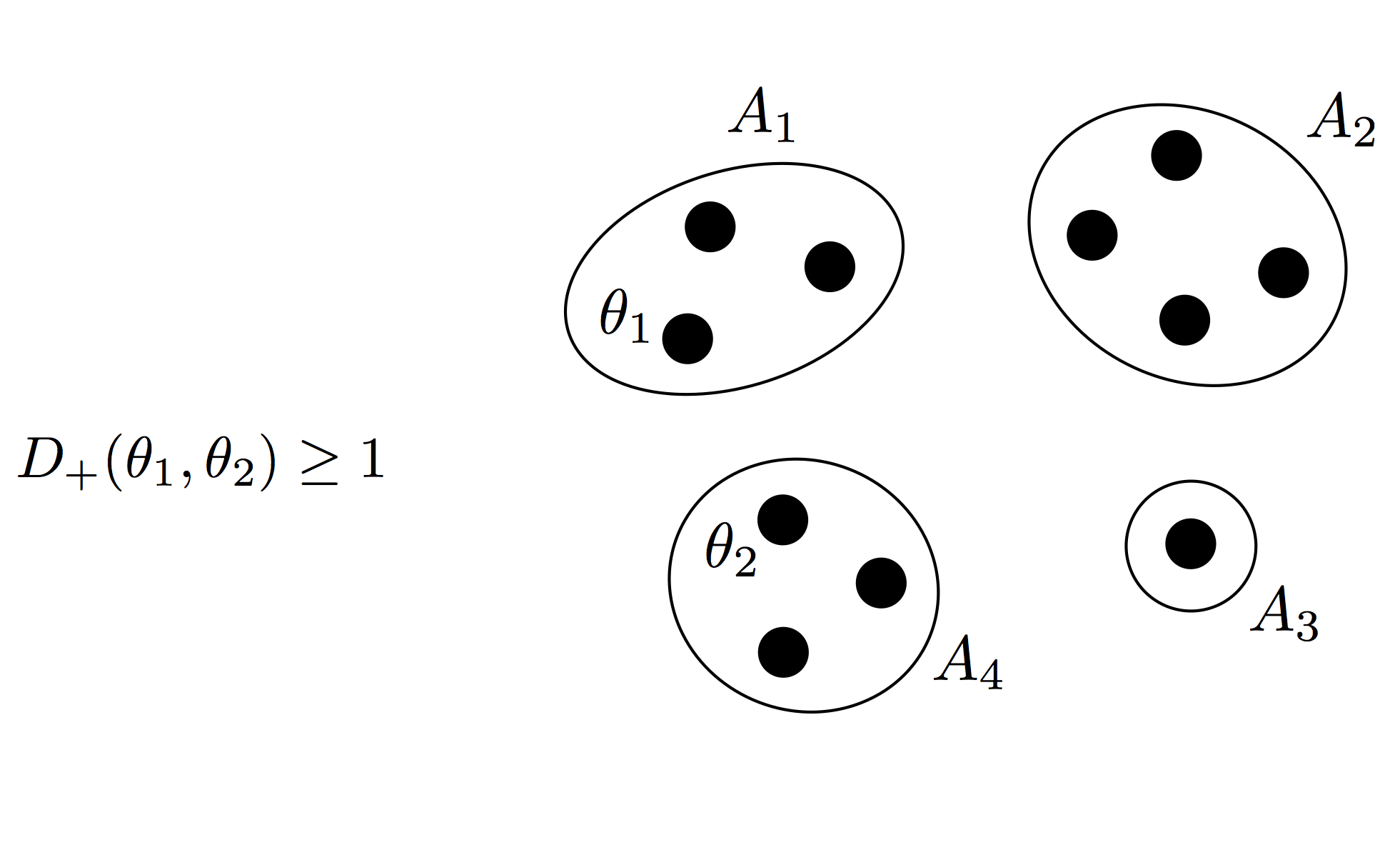}
  \caption{Finest partition: To determine which communities can be recovered in the SBM $\gs(n,p,Q)$, embed each community with its community profile $\theta_i=(PQ)_i$ in $\mR_+^k$ and find the partition of $\theta_1,\dots,\theta_k$ into the largest number of subsets that are at CH-divergence at least 1 from each other.}
  \label{finest-partition}
\end{figure}


\begin{theorem}(See Theorem \ref{thm2})
\begin{itemize}
\item Exact recovery is information-theoretically solvable in the stochastic block model $\gs(n,p,Q)$ for a partition $[k] = \sqcup_{s=1}^t A_s$ if and only if for all $i$ and $j$ in different subsets of the partition,
\begin{align}
\dd ((PQ)_i , (PQ)_j) \geq 1, \label{d1}
\end{align}
In particular, exact recovery is information-theoretically solvable in $\gs(n,p,Q)$ if and only if $\min_{i,j \in [k], i \neq j} \dd ((PQ)_i || (PQ)_j) \geq 1$.
\item The {\tt Degree-profiling} algorithm (see Section \ref{pt2}) recovers the finest partition with probability $1-o_n(1)$ and runs in $o(n^{1+\epsilon})$ time for all $\epsilon>0$. In particular, exact recovery is efficiently solvable whenever it is information-theoretically solvable. 
\end{itemize} 
\end{theorem}
This theorem assumes that the entries of $Q$ are non-zero, see Remark \ref{qzero} for zero entries. To achieve this result we rely on a two step procedure. First an algorithm is developed to recover all but a vanishing fraction of nodes --- this is the main focus of our partial recovery result next discussed --- and then a procedure is used to ``clean up'' the leftover graphs using the node degrees of the preliminary classification. This turns out to be much more efficient than aiming for an algorithm that directly achieves exact recovery. This strategy was already used in \cite{abh} for the two-community case, and appeared also in earlier works such as \cite{dyer,alon-k}. The problem is much more involved here as no algorithm is known to ensure partial recovery in the general SBM, and as classifying the nodes based on their degrees requires solving a general hypothesis testing problem for the degree-profiles in the SBM (rather than evaluating tail events of Binomial distributions). The latter part reveals the CH-divergence as the threshold for exact recovery. We next present our result for partial recovery.\\

\noindent
{\bf Partial recovery.} We obtain an algorithm that recovers the communities with an accuracy bound that tends to 1 when the average degree of the nodes gets large, and which runs in quasi-linear time. 

\begin{theorem}[See Theorem \ref{thm1}]
Given any $k\in \mathbb{Z}$, $p\in (0,1)^k$ with $|p|=1$, and symmetric matrix $Q$ with no two rows equal, let $\lambda$ be the largest eigenvalue of $PQ$, and $\lambda'$ be the eigenvalue of $PQ$ with the smallest nonzero magnitude.
If the following signal-to-noise ratio (SNR) $\rho$ satisfies
\begin{align}
\rho:=\frac{|\lambda'|^2}{\lambda}>4,\\
\lambda^7<(\lambda')^8,\\
4\lambda^3<(\lambda')^4,
\end{align} 
then for some $\e=\e(\lambda,\lambda')$ and $C=C(p,Q)>0$, the algorithm {\tt Sphere-comparison} (see Section \ref{pt1}) detects with high probability communities in graphs drawn from $\gss(n,p,Q)$ with accuracy  
\begin{align}
1- \frac{4ke^{-\frac{C \rho}{16k}}}{1-e^{-\frac{C\rho}{16k}\left(\frac{(\lambda')^4}{\lambda^3}-1\right)}} ,
\end{align}
provided that the above is larger than $1-\frac{\min_i p_i}{2\ln(4k)}$, and runs in $O(n^{1+\epsilon})$ time. Moreover, $\e$ can be made arbitrarily small with $8\ln (\lambda\sqrt{2}/|\lambda'|)/\ln(\lambda)$, and $C(p,\alpha Q)$ is independent of $\alpha$.
\end{theorem}
We next detail what previous theorem gives in the case of $k$ symmetric clusters.  
\begin{corollary}
Consider the $k$-block symmetric case. In other words, $p_i=\frac{1}{k}$ for all $i$, and $Q_{i,j}$ is $\alpha$ if $i=j$ and $\beta$ otherwise. The vector whose entries are all $1$s is an eigenvector of $PQ$ with eigenvalue $\frac{\alpha+(k-1)\beta}{k}$, and every vector whose entries add up to $0$ is an eigenvector of $PQ$ with eigenvalue $\frac{\alpha-\beta}{k}$. So, $\lambda=\frac{\alpha+(k-1)\beta}{k}$ and $\lambda'=\frac{\alpha-\beta}{k}$
and 
\begin{align}
\rho > 4 \quad \Leftrightarrow \quad \frac{(a-b)^2}{k(a+(k-1)b)} >4,
\end{align}
which is the signal-to-noise ratio appearing in the conjectures on the detection threshold for multiple blocks \cite{decelle,mossel-sbm}. 
Then, as long as $k(\alpha+(k-1)\beta)^7<(\alpha-\beta)^8$ and $4k(\alpha+(k-1)\beta)^3<(\alpha-\beta)^4$, 
there exist a constant $c>0$ (see Corollary \ref{full-coro} for details on $c$) such that {\tt Sphere-comparison} detects communities, and the accuracy is $$1-O(e^{-c(\alpha-\beta)^2/(k(\alpha+(k-1)\beta)}))$$ for sufficiently large $((\alpha-\beta)^2/(k(\alpha+(k-1)\beta)))$. 
\end{corollary}
The following is an important consequence of previous theorem, as it shows that {\tt Sphere-comparison} achieves almost exact recovery when the entires of $Q$ are amplified. 
\begin{corollary}
For any $k\in \mathbb{Z}$, $p\in (0,1)^a$ with $|p|=1$, and symmetric  matrix $Q$ with no two rows equal, there exist $\epsilon(\delta)=O(1/\ln(\delta))$ and constant $c_1>0$ such that for all sufficiently large $\delta$ there exists an algorithm ({\tt Sphere-comparison}) that detects communities in graphs drawn from $G(p,\delta Q,n)$ with accuracy at least $1-O_\delta(e^{-c_1\delta})$ in $O_n(n^{1+\epsilon(\delta)})$ time for all sufficiently large $n$.
\end{corollary}


\subsection{Information theoretic interpretation of the results}\label{it-inter}
We give in this section an interpretation of Theorem \ref{thm2} related to Shannon's channel coding theorem in information theory. At a high level, clustering the SBM is similar to reliably decoding a codeword on a channel which is non-conventional in information theory. The channel inputs are the nodes' community assignments and the channel outputs are the network edges. 
We next show that this analogy is more than just high-level: reliable communication on this channel is equivalent to exact recovery, and Theorem \ref{thm2} shows that the ``clustering capacity'' is obtained from the CH-divergence of channel-kernel $PQ$, which is an $f$-divergence like the KL-divergence governing the communication capacity.

Consider the problem of transmitting a string of $n$ $k$-ary information bits on a memoryless channel. Namely, let $X_1,\dots,X_n$ be i.i.d.\ from a distribution $p$ on $[k]$, the input distribution, and assume that we want to transmit those $k$-ary bits on a memoryless channel, whose one-time probability transition is $W$. This requires using a code, which embeds\footnote{This embedding is injective.} the vector $X^n=(X_1,\dots,X_n)$ into a larger dimension vector $U^N=(U_1,\dots,U_N)$, the codeword ($N \geq n$), such that the corrupted version of $U^N$ that the memoryless channel produces, say $Y^N$, still allows recovery of the original $U^N$ (hence $X^n$) with high probability on the channel corruptions. In other words, a code design provides the map $C$ from $X^n$ to $U^N$ (see Figure \ref{fig-c}), and a decoding map that allows recovery of $X^n$ from $Y^N$ with a vanishing error probability (i.e., reliable communication). 

Of course, if $n=N$, the encoder $C$ is just a one-to-one map, and there is no hope of defeating the corruptions of the channel $W$, unless this one is deterministic to start with. The purpose of the channel coding theorem is to quantify the best tradeoffs between $n$, $N$ and the amount of randomness in $W$, for which one can reliably communicate. When the channel is fixed and memoryless, $N$ can grow linearly with $n$, and defining the code rate by $R=n/N$, Shannon's coding theorem tells us that $R$ is achievable (i.e., there exists an encoder and decoder that allow for reliable communication at that rate) {\it if and only if}
\begin{align}
R < \max_p I(p,W),
\end{align}
where $I(p,W)$ is the mutual information of the channel $W$ for the input distribution $p$, defined as 
\begin{align}
I(p,W) = D(p \circ W||p \times pW) = \sum_{x,y} p(x) W(y|x)  \log \frac{p(x) W(y|x)}{p(x) \sum_u p(u)W(y|u)}.
\end{align}
Note that the channel capacity $\max_p I(p,W)$ is expressed in terms of the the Kullback-Leibler divergence (relative entropy) between the joint and product distribution of the channel. 

\begin{figure}[h]
\centering
\begin{subfigure}{.5\textwidth}
  \centering
  \includegraphics[width=1\linewidth]{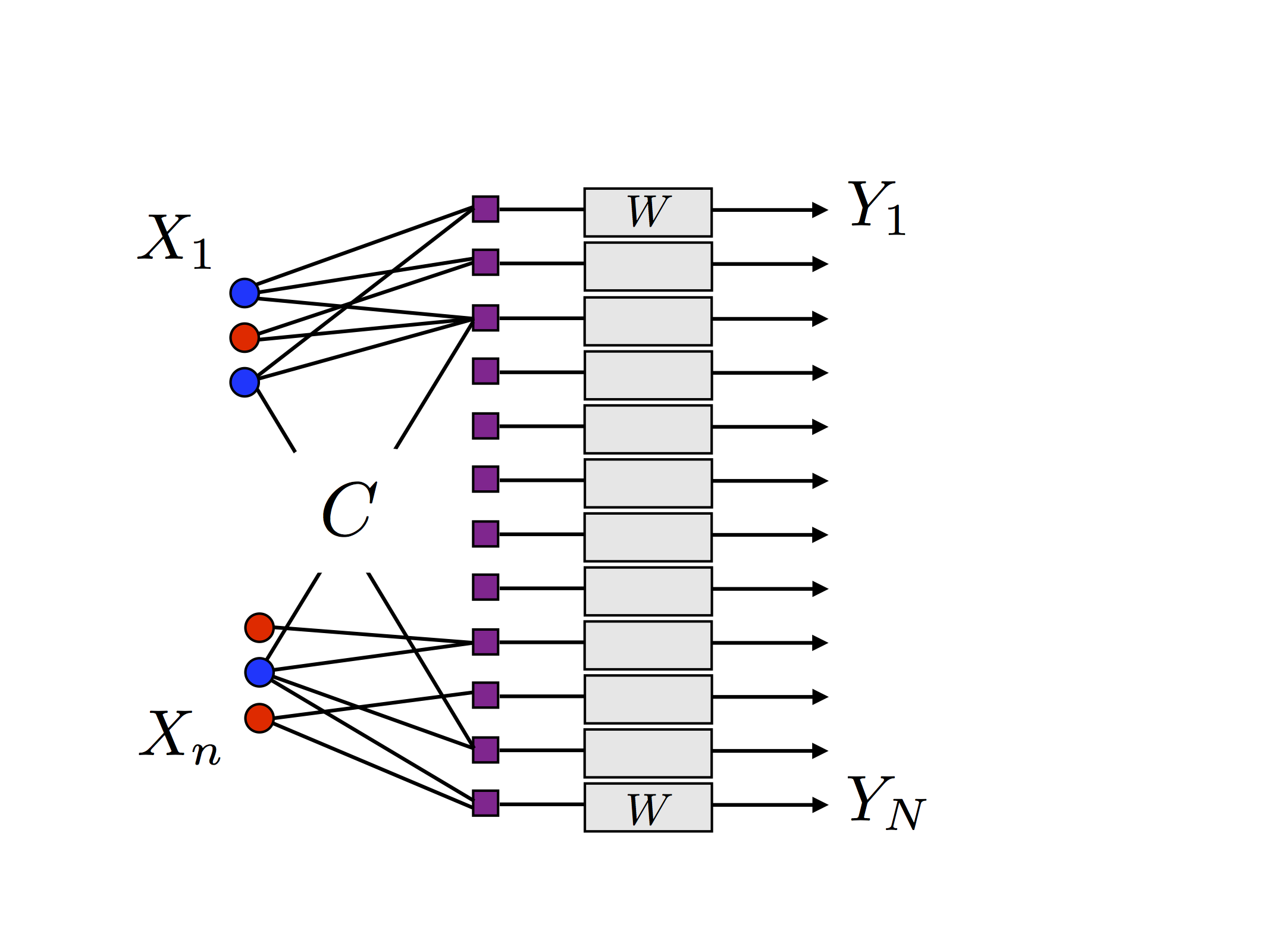}
  \caption{An encoder $C$ for data transmission.}
  \label{fig-c}
\end{subfigure}%
\begin{subfigure}{.5\textwidth}
  \centering
  \includegraphics[width=1\linewidth]{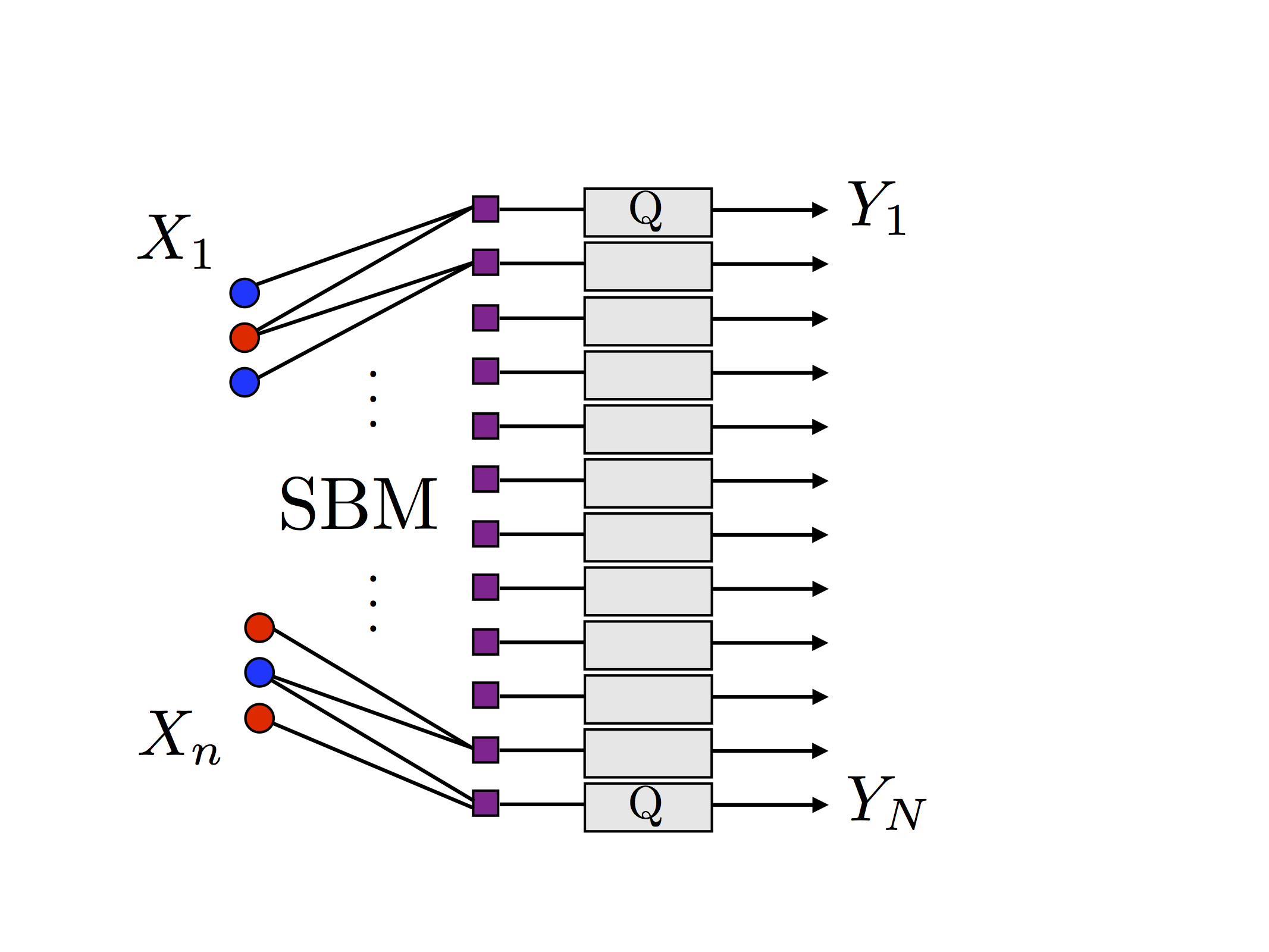}
  \caption{The SBM encoder for network modelling.}
  \label{fig-sbm}
\end{subfigure}
\caption{Clustering over the SBM can be related to channel coding over a discrete memoryless channel, for a different type of encoder and one-time channel.}
\label{double-fig}
\end{figure}

We now explain how this relates to our Theorem \ref{thm2}. Clustering the SBM can be cast as a decoding problem on a channel similar to the above. The $n$ $k$-ary information bits $X^n$ represent the community assignments to the $n$ nodes in the network. As for channel coding, these are assumed to be i.i.d.\ under some prior distribution $p$ on $[k]$.
However, clustering has several important distinctions with coding. First of all, we do not have degree of freedom on the encoder $C$. The encoder is part of the model, and in the SBM $C$ takes all possible ${n \choose 2}$ pair of information bits. In other words, the SBM corresponds to a specific encoder which has only degree 2 on the check-nodes (the squared nodes in Figure \ref{fig-sbm}) and for which $N={n \choose 2}$. Next, as in channel coding, the SBM assumes that the codeword is corrupted from a memoryless channel, which takes the two selected $k$-ary bits and maps them to an edge variable (presence or absence of edge) with a channel $W$ defined by the connectivity matrix:
\begin{align} 
W(1|x_1,x_2)&= q_{x_1,x_2},\\
W(0|x_1,x_2)&= 1-q_{x_1,x_2},
\end{align}
where $q$ scales with $n$ here. Hence, the SBM can be viewed as a specific encoder on a memoryless channel defined by the connectivity matrix $q$. We removed half of the degrees of freedom from channel coding (i.e., the encoder and $p$ are fixed), but the goal of clustering is otherwise similar to channel coding: design a decoding map that recovers the information $k$-ary bits $X^n$ from the network $Y^N$ with a vanishing error probability. In particular, exact recovery is equivalent to reliable communication.

A naive guess would be that some mutual information derived from the input distribution $p$ and the channel induced from $q$ could give the fundamental tradeoffs, as for channel coding. However, this is where the difference between coding and clustering is important. An encoder that achieves capacity in the traditional setting is typically ``well spread,'' for example, like a random code which picks each edge in the bipartite graph of Figure \ref{fig-c} with probability one half. The SBM encoder, instead, is structured in a very special way, which may not be well suited for communication purposes\footnote{It corresponds for example to a 2-right degree LDGM code in the case of the symmetric two-community SBM, a code typically not used for communication purposes.}. This makes of course sense as the formation of a real network should have nothing to do with the design of an engineering system. Note also that the code rate in the SBM channel is fixed to $R=\frac{n}{{n \choose 2}} \sim \frac{2}{n}$, 
which means that there is hope to still decode such a ``poor'' code, even on a very noisy channel.  

Theorem \ref{thm2} shows that indeed a similar phenomenon to channel coding takes place for clustering. Namely, there exists a notion of ``capacity,'' governed not by KL-divergence but the CH-divergence introduced in Section \ref{main-res}. The resulting capacity captures if reliable communication is possible or not. The relevant regime is for $q$ that scales like $\ln(n)Q/n$, and the theorem says that it is possible to decode the inputs (i.e., to recover the communities) if and only if 
\begin{align}
1 \leq J(p,Q),
\end{align} 
where 
\begin{align}
J(p,Q)= \min_{i\neq j} \dd((pQ)_i,(pQ)_j).
\end{align} 
Note again the difference with the channel coding theorem: here we cannot optimize over $p$ (since the community sizes are not a design parameter), and the rate $R$ is fixed. One could change the latter requirement, defining a model where the information about the edges is only revealed at a given rate, in which case the analogy with Shannon's theorem can be made even stronger (see for example \cite{abbs}.) 

The conclusiong is that we can characterize the fundamental limit for clustering, with a sharp transition governed by a measure of the channel ``noisiness,'' that is related to the KL-divergence used in the channel coding theorem. This is due to the hypothesis testing procedures underneath both frameworks (see Section \ref{testing}). 
Defining
\begin{align} 
D_t(\mu,\nu):= \sum_{x \in [k]} \left(  t\mu(x) + (1-t)\nu(x)- \mu(x)^t \nu(x)^{1-t}  \right) 
\end{align}
we have that 
\begin{itemize}
\item $D_t$ is an $f$-divergence, that is, it can be expressed as $\sum_x \nu(x) f(\mu(x)/\nu(x))$, where $f(x)=1-t+tx-x^t$, which is convex. The family of $f$-divergences were defined in \cite{csiszar-f,morimoto,ali} and shown to have various common properties when $f$ is convex. Note that the KL-divergence is also an $f$-divergence for the convex function $f(x)=x\ln(x)$,
\item $D_{1/2}(\mu,\nu)=\frac{1}{2} \| \sqrt{\mu} - \sqrt{\nu} \|_2^2$ is the Hellinger divergence (or distance), in particular, this is the maximizer for the case of two symmetric communities, recovering the expression $\frac{1}{2}(\sqrt{a}-\sqrt{b})^2$ obtained in \cite{abh,mossel-consist},
\item $D_t(\mu,\nu)=t \bar{\mu} - (1-t) \bar{\nu} - e^{-D_t(\mu||\nu)}$, where $D_t(\cdot || \cdot)$ is the R\'enyi divergence, and the maximization over $t$ of this divergence is the Chernoff divergence.
\end{itemize}
As a result, $D_+$ can be viewed as a generalization of the Hellinger and Chernoff divergences. We hence call it the Chernoff-Hellinger (CH) divergence. Theorem \ref{thm2} gives hence an operational meaning to $\dd$ with the community recovery problem.  It further shows that the limit can be efficiently achieved.

\section{Proof Techniques and Algorithms}
\subsection{Partial recovery and the {\tt Sphere-comparison} algorithm}\label{pt1}
The first key observation used to classify graphs' vertices is that if $v$ is a vertex in a graph drawn from $\gss(n,p,Q)$ then for all small $r$ the expected number of vertices in community $i$ that are $r$ edges away from $v$ is approximately $e_i\cdot(PQ)^re_{\sigma_v}$. So, we define: 

\begin{definition}\label{def-n1}
For any vertex $v$, let $N_r(v)$ be the set of all vertices with shortest path to $v$ of length $r$. If there are multiple graphs that $v$ could be considered a vertex in, let $N_{r[G]}(v)$ be the set of all vertices with shortest paths in $G$ to $v$ of length $r$.
\end{definition}

We also refer to the vector with $i$-th entry equal to the number of vertices in $N_r(v)$ that are in community $i$ as $N_r(v)$. One could determine $e_{\sigma_v}$ given $(PQ)^re_{\sigma_v}$, but using $N_r(v)$ to approximate that would require knowing how many of the vertices in $N_r(v)$ are in each community. So, we attempt to get information relating to how many vertices in $N_r(v)$ are in each community by checking how it connects to $N_{r'}(v')$ for some vertex $v'$ and integer $r'$. The obvious way to do this would be to compute the cardinality of their intersection. Unfortunately, whether a given vertex in community $i$ is in $N_r(v)$ is not independent of whether it is in $N_{r'}(v')$, which causes the cardinality of $|N_r(v)\cap N_{r'}(v')|$ to differ from what one would expect badly enough to disrupt plans to use it for approximations.

In order to get around this, we randomly assign every edge in $G$ to a set $E$ with probability $c$. We hence define the following. 

\begin{definition}
For any vertices $v, v'\in G$, $r,r'\in \mathbb{Z}$, and subset of $G$'s edges $E$, let $N_{r,r'[E]}(v\cdot v')$ be the number of pairs of vertices $(v_1,v_2)$ such that $v_1\in N_{r[G\backslash E]}(v)$, $v_2\in N_{r'[G\backslash E]}(v')$, and $(v_1,v_2)\in E$.
\end{definition}

Note that $E$ and $G\backslash E$ are disjoint; however, $G$ is sparse enough that even if they were generated independently a given pair of vertices would have an edge between them in both with probability $O(\frac{1}{n^2})$. So, $E$ is approximately independent of $G\backslash E$. Thus, for any $v_1\in N_{r[G/E]}(v)$ and $v_2\in N_{r'[G/E]}(v')$, $(v_1,v_2)\in E$ with a probability of approximately $cQ_{\sigma_{v_1},\sigma_{v_2}}/n$. As a result,

\begin{align*}
N_{r,r'}[E](v\cdot v')&\approx N_{r[G\backslash E]}(v)\cdot \frac{cQ}{n} N_{r'[G\backslash E]}(v')\\
&\approx ((1-c)PQ)^re_{\sigma_v}\cdot \frac{cQ}{n} ((1-c)PQ)^{r'}e_{\sigma_{v'}}\\
&=c(1-c)^{r+r'}e_{\sigma_v}\cdot Q(PQ)^{r+r'}e_{\sigma_{v'}}/n
\end{align*}

 Let $\lambda_1,...,\lambda_h$ be the distinct eigenvalues of $PQ$, ordered so that $|\lambda_1|\ge|\lambda_2|\ge...\ge|\lambda_h|\ge 0$. Also define $\eta$ so that $\eta=h$ if $\lambda_h\ne 0$ and $\eta=h-1$ if $\lambda_h=0$. If $W_i$ is the eigenspace of $PQ$ corresponding to the eigenvalue $\lambda_i$, and $P_{W_i}$ is the projection operator on to $W_i$, then 
\begin{align*}
N_{r,r'}[E](v\cdot v')&\approx c(1-c)^{r+r'}e_{\sigma_v}\cdot Q(PQ)^{r+r'}e_{\sigma_{v'}}/n\\
&=\frac{c(1-c)^{r+r'}}{n} \left(\sum_i P_{W_i}(e_{\sigma_v})\right)\cdot Q(PQ)^{r+r'}\left(\sum_j P_{W_j}(e_{\sigma_{v'}})\right)\\
&=\frac{c(1-c)^{r+r'}}{n} \sum_{i,j}P_{W_i}(e_{\sigma_v})\cdot Q(PQ)^{r+r'}P_{W_j}(e_{\sigma_{v'}})\\
&=\frac{c(1-c)^{r+r'}}{n} \sum_{i,j}P_{W_i}(e_{\sigma_v})\cdot P^{-1}(\lambda_j)^{r+r'+1}P_{W_j}(e_{\sigma_{v'}})\\
&=\frac{c(1-c)^{r+r'}}{n} \sum_{i}\lambda_i^{r+r'+1} P_{W_i}(e_{\sigma_v})\cdot P^{-1}P_{W_i}(e_{\sigma_{v'}})
\end{align*}
where the final equality holds because for all $i\ne j$, 
\begin{align*}
\lambda_i  P_{W_i}(e_{\sigma_v})\cdot P^{-1} P_{W_j}(e_{\sigma_{v'}})&=(PQ  P_{W_i}(e_{\sigma_v}))\cdot P^{-1} P_{W_j}(e_{\sigma_{v'}})\\
&= P_{W_i}(e_{\sigma_v})\cdot Q P_{W_j}(e_{\sigma_{v'}})\\
&= P_{W_i}(e_{\sigma_v})\cdot P^{-1} \lambda_j P_{W_j}(e_{\sigma_{v'}}),
\end{align*}
and since $\lambda_i\ne \lambda_j$, this implies that $P_{W_i}(e_{\sigma_v})\cdot P^{-1} P_{W_j}(e_{\sigma_{v'}})=0$.

That implies that one can approximately solve for $P_{W_i}e_{\sigma_v}\cdot P^{-1}P_{W_i}e_{\sigma_{v'}}$ given $N_{r,r'+j}(v\cdot v')$ for all $0\le j<\eta$. Of course, this requires $r$ and $r'$ to be large enough such that \[\frac{c(1-c)^{r+r'}}{n} \lambda_i^{r+r'+1} P_{W_i}(e_{\sigma_v})\cdot P^{-1}P_{W_i}(e_{\sigma_{v'}})\] is large relative to the error terms for all $i\le \eta$. At a minimum, that requires that $|(1-c)\lambda_i|^{r+r'+1}=\omega(n)$, so \[r+r' >\log( n)/\log((1-c)|\lambda_{\eta}|).\] On the flip side, one also needs \[r,r'<\log(n)/\log ((1-c)\lambda_1)\] because otherwise the graph will start running out of vertices before one gets $r$ steps away from $v$ or $r'$ steps away from $v'$.

Furthermore, for any $v$ and $v'$, 
\begin{align*}
0 &\le P_{W_i}(e_{\sigma_v}-e_{\sigma_{v'}})\cdot P^{-1}P_{W_i}(e_{\sigma_v}-e_{\sigma_{v'}})\\
&=P_{W_i}e_{\sigma_v}\cdot P^{-1}P_{W_i}e_{\sigma_{v}}-2P_{W_i}e_{\sigma_v}\cdot P^{-1}P_{W_i}e_{\sigma_{v'}}+P_{W_i}e_{\sigma_{v'}}\cdot P^{-1}P_{W_i}e_{\sigma_{v'}}
\end{align*}
with equality for all $i$ if and only if $\sigma_v=\sigma_{v'}$, so sufficiently good approximations of $P_{W_i}e_{\sigma_v}\cdot P^{-1}P_{W_i}e_{\sigma_{v}}, P_{W_i}e_{\sigma_v}\cdot P^{-1}P_{W_i}e_{\sigma_{v'}}$ and $P_{W_i}e_{\sigma_v'}\cdot P^{-1}P_{W_i}e_{\sigma_{v'}}$ can be used to determine which pairs of vertices are in the same community as follows. \\

\noindent
{\bf The Vertex\_comparison\_algorithm}. 
The inputs are $(v, v', r, r',E, x, c)$, where $v,v'$ are two vertices, $r, r'$ are positive integers, $E$ is a subset of $G$'s edges, $x$ is a positive real number, and $c$ is a real number between $0$ and $1$. 

The algorithm outputs a decision on whether $v$ and $v'$ are in the same community or not. It proceeds as follows. 

(1) Solve the systems of equations: 
\[\sum_i ((1-c)\lambda_i)^{r+r'+j+1}z_i(v\cdot v')=\frac{(1-c)n}{c}N_{r+j,r'[E]}(v\cdot v') \text{ for } 0\le j<\eta\]
\[\sum_i ((1-c)\lambda_i)^{r+r'+j+1}z_i(v\cdot v)=\frac{(1-c)n}{c}N_{r+j,r'[E]}(v\cdot v) \text{ for } 0\le j<\eta\]
\[\sum_i ((1-c)\lambda_i)^{r+r'+j+1}z_i(v'\cdot v')=\frac{(1-c)n}{c}N_{r+j,r'[E]}(v'\cdot v') \text{ for } 0\le j<\eta\]

(2) If $\exists i: z_i(v\cdot v)-2z_i(v\cdot v')+z_i(v'\cdot v')> 5(2x(\min p_j)^{-1/2}+x^2)$ then conclude that $v$ and $v'$ are in different communities. Otherwise, conclude that $v$ and $v'$ are in the same community.\\

One could generate a reasonable classification based solely on this method of comparing vertices (with an appropriate choice of the parameters, as later detailed). However, that would require computing $N_{r,r'[E]}(v\cdot v)$ for every vertex in the graph with fairly large $r+r'$, which would be slow. Instead, we use the fact that for any vertices $v$, $v'$, and $v''$ with $\sigma_v=\sigma_{v'}\ne\sigma_{v''}$, 
\begin{align*}
&P_{W_i}e_{\sigma_{v'}}\cdot P^{-1}P_{W_i}e_{\sigma_{v'}}-2P_{W_i}e_{\sigma_v}\cdot P^{-1}P_{W_i}e_{\sigma_{v'}}+P_{W_i}e_{\sigma_v}\cdot P^{-1}P_{W_i}e_{\sigma_{v}}=0\\
&\le P_{W_i}e_{\sigma_{v''}}\cdot P^{-1}P_{W_i}e_{\sigma_{v''}}-2P_{W_i}e_{\sigma_v}\cdot P^{-1}P_{W_i}e_{\sigma_{v''}}+P_{W_i}e_{\sigma_v}\cdot P^{-1}P_{W_i}e_{\sigma_{v}}
\end{align*}
for all $i$, and the inequality is strict for at least one $i$. So, subtracting $P_{W_i}e_{\sigma_v}\cdot P^{-1}P_{W_i}e_{\sigma_{v}}$ from both sides gives us that
\[P_{W_i}e_{\sigma_{v'}}\cdot P^{-1}P_{W_i}e_{\sigma_{v'}}-2P_{W_i}e_{\sigma_v}\cdot P^{-1}P_{W_i}e_{\sigma_{v'}}\le P_{W_i}e_{\sigma_{v''}}\cdot P^{-1}P_{W_i}e_{\sigma_{v''}}-2P_{W_i}e_{\sigma_v}\cdot P^{-1}P_{W_i}e_{\sigma_{v''}}\]
for all $i$, and the inequality is still strict for at least one $i$.

So, given a representative vertex in each community, we can determine which of them a given vertex, $v$, is in the same community as without needing to know the value of $P_{W_i}e_{\sigma_v}\cdot P^{-1}P_{W_i}e_{\sigma_{v}}$ as follows.\\

\noindent
{\bf The Vertex\_classification\_algorithm}. The inputs are $(v[],v', r,r',E,x,c)$, where $v[]$ is a list of vertices, $v'$ is a vertex, $r,r'$ are positive integers, $E$ is a subset of $G$'s edges, $x$ is a positive real number, and $c$ is a real number between $0$ and $1$. It is assumed that $z_i(v[\sigma]\cdot v[\sigma] )$ satisfying 
\[\sum_i ((1-c)\lambda_i)^{r+r'+j+1}z_i(v[\sigma]\cdot v[\sigma])=\frac{(1-c)n}{c}N_{r+j,r'[E]}(v[\sigma]\cdot v[\sigma]) \text{ for } 0\le j<\eta\]
have already been computed for every $v[\sigma]\in v[]$. 

The algorithm is supposed to output $\sigma$ such that $v'$ is in the same community as $v[\sigma]$. It works as follows. 

(1) For each $\sigma$ solve the system of equations  \[\sum_i ((1-c)\lambda_i)^{r+r'+j+1}z_i(v[\sigma]\cdot v')=\frac{(1-c)n}{c}N_{r+j,r'[E]}(v[\sigma]\cdot v') \text{ for } 0\le j<\eta\]

(2) If there exists a unique $\sigma$ such that for all $\sigma'\ne \sigma$ and all $i$, \[z_i(v[\sigma]\cdot v[\sigma])-2z_i(v[\sigma]\cdot v')\le z_i(v[\sigma']\cdot v[\sigma'])-2z_i(v[\sigma']\cdot v') +\frac{19}{3}\cdot (2x(\min p_j)^{-1/2}+x^2)\] then conclude that $v'$ is in the same community as $v[\sigma]$.

(3) Otherwise, Fail.\\

This runs fairly quickly if $r$ is large and $r'$ is small because the algorithm only requires focusing on $N_{r'}(v')$ vertices. This leads to the following plan for partial recovery. First, randomly select a set of vertices that is large enough to contain at least one vertex from each community with high probability. Next, compare all of the selected vertices in an attempt to determine which of them are in the same communities. Then, pick one in each community. After that, use the algorithm above to attempt to determine which community each of the remaining vertices is in. As long as there actually was at least one vertex from each community in the initial set and none of the approximations were particularly bad, this should give a reasonably accurate classification. \\

\noindent
{\bf The Unreliable\_graph\_classification\_algorithm}. The inputs are $(G,c,m,\epsilon,x)$, where $G$ is a graph, $c$ is a real number between $0$ and $1$, $m$ is a positive integer, $\epsilon$ is a real number between $0$ and $1$, and $x$ is a positive real number. 

The algorithm outputs an alleged list of communities for $G$. It works as follows.

(1) Randomly assign each edge in $G$ to $E$ independently with probability $c$.

(2) Randomly select $m$ vertices in $G$, $v[0],...,v[m-1]$.

(3) Set $r=(1-\frac{\epsilon}{3})\log n/\log ((1-c)\lambda_1)-\eta$ and $r'=\frac{2\epsilon}{3}\cdot \log n/\log((1-c) \lambda_1)$

(4) Compute $N_{r''[G\backslash E]}(v[i])$ for each $r''<r+\eta$ and $0\le i<m$.

(5) Run $\text{Vertex\_comparison\_algorithm}(v[i],v[j],r,r',E,x)$ for every $i$ and $j$

(6) If these give results consistent with some community memberships which indicate that there is at least one vertex in each community in $v[]$, randomly select one alleged member of each community $v'[\sigma]$. Otherwise, fail.

(7) For every $v''$ in the graph, compute $N_{r''[G\backslash E]}(v'')$ for each $r''<r'$. Then, run\newline $\text{Vertex\_classification\_algorithm}(v'[],v'', r,r',E,x)$ in order to get a hypothesized classification of $v''$.

(8) Return the resulting classification.\\

The risk that this randomly gives a bad classification due to a bad set of initial vertices can be mitigated as follows. First, repeat the previous classification procedure several times. Next, discard any classification that differs too badly from the majority. Assuming that the procedure gives a good classification more often than not, this should eliminate any really bad classification. Finally, average the remaining classifications together. This last procedure completes the {\tt Sphere comparison} algorithm.\\

\noindent
{\bf The Reliable\_graph\_classification\_algorithm (i.e., {\tt Sphere comparison})}. The inputs are $(G,c,m,\epsilon,x,T(n))$, where $G$ is a graph, $c$ is a real number between $0$ and $1$, $m$ is a positive integer, $\epsilon$ is a real number between $0$ and $1$, $x$ is a positive real number, and $T$ is a function from the positive integers to itself. 

The algorithm outputs an alleged list of communities for $G$. It works as follows.

(1) Run $\text{Unreliable\_graph\_classification\_algorithm}(G,c,m,\epsilon,x)$ $T(n)$ times and record the resulting classifications.

(2) Discard any classification that has greater than $$4ke^{-\frac{(1-c)x^2\lambda_{\eta}^2\min p_i}{16\lambda_1 k(1+x)}}/(1-e^{-\frac{(1-c)x^2\lambda_{\eta}^2\min p_i}{16\lambda_1 k(1+x)}\cdot((\frac{(1-c)\lambda_{\eta}^4}{4\lambda_1^3})-1)})$$  disagreement with more than half of the other classifications. In this step, define the disagreement between two classifications as the minimum disagreement over all bijections between their communities.

(3) Let $\{\sigma[i]\}$ be the remaining classifications. For each vertex $v\in G$, randomly select some $i$ and assert that $\sigma_v$ is the $j$ that maximizes $|\{v': \sigma[1]_{v'=}j\}\cap\{v':\sigma[i]_{v'}=\sigma[i]_v\}|$. In other words, assume that $\sigma[i]$ classifies $v$ correctly and then translate that to a community of $\sigma[1]$ by assuming the communities of $\sigma[i]$ correspond to the communities of $\sigma[1]$ that they have the greatest overlap with.

(4) Return the resulting combined classification.\\

If the conditions of theorem $2$ are satisfied, then there exists $x$ such that for all sufficiently small $c$, $$\text{Reliable\_graph\_classification\_algorithm}(G,c,\ln(4k)/\min p_i,\epsilon,x,\ln n)$$ classifies at least 
\begin{align}
1- \frac{4ke^{-\frac{C \rho}{16k}}}{1-e^{-\frac{C\rho}{16k}\left(\frac{(\lambda')^2}{4\lambda^2}\rho-1\right)}}
\end{align}
of $G$'s vertices correctly with probability $1-o(1)$ and it runs in $O(n^{1+\epsilon})$ time.

\subsection{Exact recovery and the {\tt Degree-profiling} algorithm}\label{pt2}
With our previous result achieving almost exact recovery of the nodes, we are in a position to complete the exact recovery via a procedure that performs local improvements on the rough solution. While, the exact recovery requirement is rather strong, we show that it benefits from a phase transition, as opposed to almost exact recovery, which allows to benchmark algorithms on a sharp limit (see Introduction). 

Our analysis of exact recovery relies on the fact that the probability distribution of the numbers of neighbors a given vertex has in each community is essentially a multivariable Poisson distribution. We hence investigate an hypothesis problem (see Section \ref{testing}), where a node in the SBM graph with known clusters (up to $o(n)$ errors due to our previous results) is taken and re-classified based on its degree profile, i.e., on the number of neighbors it has in each community. This requires testing between $k$ multivariate Poisson distributions of different means $m_1,\dots,m_k \in \mZ_+^k$, where $m_i=\ln(n)\theta_i$ for $\theta_i=(PQ)_i$, $i \in [k]$. The error probability of the optimal test depends on the degree of overlap between any two of these Poisson distributions, which we show is either $o(\frac{1}{n})$ or $\omega(\frac{1}{n})$. This is where the CH-divergence emerges as the exponent for the error probability. It is captured by the following sharp estimate derived in Section \ref{testing}, where $\mathcal{P}_{c}$ denotes the Poisson distribution of mean $c$.
\begin{theorem}(See Lemma \ref{hell-expo}.)
For any $\theta_1, \theta_2 \in (\mR_+\setminus \{0\})^k$ with $\theta_1 \neq \theta_2$ and $p_1,p_2 \in \mR_+\setminus \{0\}$, 
\begin{align*}
& \sum_{x \in \mZ_+^k} \min(\mathcal{P}_{\ln(n) \theta_1}(x) p_{1} ,  \mathcal{P}_{\ln(n) \theta_2}(x) p_{2}) = \Theta\left(n^{- \dd(\theta_1,\theta_2) - o(1)} \right),
\end{align*}
where $\dd(\theta_1,\theta_2)$ is the CH-divergence as defined in \eqref{h-div1}.
\end{theorem}
Using this result, we show that depending on the parameters of the SBM, the error probability of the optimal test is either $o(\frac{1}{n})$ or $\omega(\frac{1}{n})$ depending on $\min_{i<j} \dd(\theta_i,\theta_j)$. If the error probability is $\omega(\frac{1}{n})$ then any method of distinguishing between vertices in those two communities must fail with probability $\omega(\frac{1}{n})$, so any possible algorithm attempting to distinguish between them must misclassify at least one vertex with probability $1-o(1)$. On the other hand, if the degree of overlap between all communities we are trying to distinguish between is $o(1/n)$ then with probability $1-o(1)$ one could correctly classify any vertex in the graph if one knew what community each of its neighbors was in. There exists $\delta$ such that attempting to classify a vertex based on classifications of its neighbors that are wrong with probability $x$ results in a probability of misclassifying the vertex that is only $n^{\delta x}$ times as high as it would be if they were all classified correctly. Based on this, the obvious approach to exact recovery would be to use a partial recovery algorithm to create a preliminary classification and then attempt to determine which family of communities each vertex is in based on its neighbors' alleged communities. However, the standard partial recovery algorithm has a constant error rate, so this proceedure's output would have an error rate $n^c$ times as large as if each vertex were being classified based on its neighbors' true communities for some $c>0$. If the degrees of overlap are only barely below $1/n$ then this would increase the error rate enough that this procedure  would misclassify at least one vertex with high probability.

Instead, we go through three successively more accurate classifications as follows. Given a partial reconstruction of the communities with an error rate that is a sufficiently low constant, one can classify vertices based on their neighbors' alleged communities with an accuracy of $1-O(n^{-c})$ for some constant $c>0$. Then one can use this classification of a vertex's neighbors to determine which family of communities it is in with an accuracy of $1-o(\frac{1}{n}\cdot n^{\delta c'n^{-c}})=1-o(1/n)$. Therefore, the resulting classification is correct with probability $1-o(1)$. 

We formulate the algorithm in an adaptive way, where we first identify which communities can be exactly recovered with the notion of ``finest partition,'' and then proceed to extract this partition. In other words, even in the case where not all communities can be exactly recovered, the algorithm may be able to fully extract a subset of the communities. Overall, the algorithm for exact recovery works as follows.\\

{\bf The {\tt Degree-profiling} algorithm.} The inputs are $(G, \gamma)$, where $G$ is a graph, and $\gamma\in [0,1]$ (see Theorem \ref{thm2} for how to set $\gamma$). The algorithm outputs an assignment of each vertex to one of the groups of communities $\{A_1,\dots,A_t\}$, where $A_1,\dots,A_t$ is the partition of $[k]$ in to the largest number of subsets such that $\dd((pQ)_i,(pQ)_j) \geq 1$ for all $i,j$ in $[k]$ that are in different subsets (i.e., the ``finest partition,'' see Firgure \ref{finest-partition}). It does the following:

(1) Define the graph $g'$ on the vertex set $[n]$ by selecting each edge in $g$ independently with probability $\gamma$, and define the graph $g''$ that contains the edges in $g$ that are not in $g'$. 

(2) Run {\tt Sphere-comparison} on $g'$ to obtain the preliminary classification $\sigma' \in [k]^n$ (see Section \ref{partial-sec}.) 

(3) Determine the edge density between each pair of alleged communities, and use this information and the alleged communities' sizes to attempt to identify the communities up to symmetry. 

(4) For each node $v \in [n]$, determine in which community node $v$ is most likely to belong to based on its degree profile computed from the preliminary classification $\sigma'$ (see Section \ref{testing}), and call it $\sigma''_v$

(5) For each node $v \in [n]$, determine in which group $A_1,\dots,A_t$ node $v$ is most likely to belong to based on its degree profile computed from the preliminary classification $\sigma''$ (see Section \ref{testing}).


\section{Overlapping communities}\label{overlap}
We now define a model that accounts for overlapping communities, we refer to it as the overlapping stochastic block model (OSBM). 
\begin{definition}
Let $n,t \in \mZ_+$, $f: \{0,1\}^t \times \{0,1\}^t \to [0,1]$ symmetric, and $p$ a probability distribution on $\{0,1\}^t$. A random graph with distribution OSBM$(n,p,f)$ is generated on the vertex set $[n]$ by drawing independently for each $v \in [n]$ the vector-labels (or user profiles) $X(v)$ under $p$, and by drawing independently for each $u,v \in [n]$, $u < v$, an edge between $u$ and $v$ with probability $f(X(u),X(v))$.
\end{definition}
\begin{example}\label{common}
One may consider $f(x,y)= \theta_g(x,y)$, where $x_i$ encodes whether a node is in community $i$ or not, and 
\begin{align}
\theta_g(x,y)= g(\langle x , y \rangle),
\end{align}
where $\langle x , y \rangle = \sum_{i=1}^t x_i y_i$ counts the number of common communities between the labels $x$ and $y$, and $g:\{0,1,\dots, t\} \to [0,1]$ is a function that maps the overlap score into probabilities ($g$ is typically increasing). 
\end{example}
\begin{example}
As a special case of the previous example, one may consider that a connection takes place between each pair of nodes as follows: 
each community (i.e., each component in the user profile) generates a connection independently with probability $q_+$ if the two nodes are in that community (i.e., if that component is 1 for both profiles), and multiple connections are equivalent to one connection. We also assume that any pair of nodes without a common community connects with probability $q_-$, so that   
\begin{align}
g(s)=
\begin{cases}
1- (1-q_+)^s, &\text{ if } s \neq 0,\\
p_-, &\text{ if } s=0.
\end{cases}
\end{align}
If we consider $q_-$ and $q_+$ to be vanishing, like $O(\log(n)/n)$, we may consider the equivalent model where
\begin{align}
g(s)&=
\begin{cases}
sq_+, &\text{ if } s \neq 0,\\
p_-, &\text{ if } s=0.
\end{cases}
\end{align}
If $t=1$, this model collapses to the usual symmetric stochastic block model with non-overlapping communities. 
\end{example}

Note that in general we can represent the OSBM as a SBM with $k=2^t$ communities, where each community represents a possible profile in $\{0,1\}^t$. For example, two overlapping communities can be modelled by assigning nodes with a single attribute $(1,0)$ and $(0,1)$ to each of the disjoint communities and nodes with both attributes $(1,1)$ to the overlap community, while nodes having none of the attributes, i.e.,  $(0,0)$, may be assigned to the null community.

Assume now that we identify community $i \in [k]$ with the profile corresponding to the binary expansion of $i-1$. The prior and connectivity matrix of the corresponding SBM are then given by 
\begin{align}
p_i&=p(b(i))\label{fmap1} \\
q_{i,j}&=f(b(i),b(j)), \label{fmap2}
\end{align}
where $b(i)$ is the binary expansion of $i-1$, and 
\begin{align}
\text{OSBM}(n,p,f) \stackrel{(d)}{=} \text{SBM}(n,p,q).
\end{align}
We can then use the results of previous sections to obtain exact recovery in the OSBM. 
\begin{corollary}
Exact recovery is solvable for the OSBM if the conditions of Theorem \ref{thm2} apply to the SBM$(n,p,q)$ with $p$ and $q$ as defined in \eqref{fmap1}, \eqref{fmap2}.  
\end{corollary}


\section{Further literature}\label{related_lit}

The stochastic block model was first introduced in \cite{holland}, and in \cite{bui,dyer} as the planted bisection model. 
For the first three decades, a major portion of the literature has focused on exact recovery, in particular on the case with two symmetric communities. The table below summarizes a partial list of works for {\bf exact recovery:}

\begin{small}
\begin{center}
  \begin{tabular}{| l |  c |c   | }
  \hline
Bui, Chaudhuri,  && \\
Leighton, Sipser '84  & min-cut method& $p = \Omega(1/n), q=o(n^{-1-4/((p+q)n)})$\\
  \hline
Dyer, Frieze '89  & min-cut via degrees & $p -q = \Omega(1)$\\
  \hline
Boppana '87  & spectral method & $(p -q)/\sqrt{p+q} = \Omega(\sqrt{\log(n)/n})$\\
  \hline
Snijders, Nowicki '97  & EM algorithm & $p -q = \Omega(1)$\\
\hline
Jerrum, Sorkin '98  &Metropolis algorithm & $p -q= \Omega(n^{-1/6+\e})$\\
   \hline
Condon, Karp '99 & augmentation algorithm & $p -q= \Omega(n^{-1/2+\e})$\\
  \hline
Carson, Impagliazzo '01  & hill-climbing algorithm & $p-q= \Omega(n^{-1/2} \log^4(n))$\\
 \hline
Mcsherry '01  & spectral method & $(p - q)/\sqrt{p} \geq \Omega (\sqrt{\log(n)/n })$\\
 \hline
Bickel, Chen '09  & N-G modularity & $(p -q)/\sqrt{p+q} = \Omega(\log(n)/\sqrt{n})$\\
 \hline
Rohe, Chatterjee, Yu '11  & spectral method & $p -q= \Omega(1)$\\
\hline
  \end{tabular}
\end{center}
\end{small}
These works display an impressive diversity of algorithms, but are mainly driven by the methodology and do not reveal the sharp behavioral transition that takes place in this model, as later shown in \cite{abh,mossel-consist} (see below). Before discussing these results, one should mention that various other works have considered recovery algorithms for multiple communities without identifying phase transitions. We refer to \cite{sbm-algos,chen-xu} for a summary of these results. In particular, \cite{chen-xu} has recently studied information-theoretic vs.\ computational tradeoffs in coarse regimes of the parameters for symmetric block models with a growing number of communities. 

Phase transition phenomena for the SBM appeared first for weak recovery. In 2010, Coja-Oghalan \cite{coja-sbm} introduced the weak-recovery problem, and obtained bounds for the constant average degree regime using a spectral algorithm. Soon after, \cite{decelle} proposed a precise picture for weak-recovery using statistical physics arguments, with a sharp threshold conjectured at $(a-b)^2 = 2(a+b)$, when $a=pn$ and $b=pn$. This has opened the door to a new series of work on the SBM driven by phase transitions. The impossibility part of the conjecture was first proved in \cite{mossel-sbm}, using a reduction to broadcasting on trees \cite{evans}, and the conjecture was fully established in 2014 with \cite{massoulie-STOC,Mossel_SBM2}. 

Recently it was realized that exact recovery also admits a phase transition pheonemon. This was set in \cite{abh}, and shortly after in \cite{mossel-consist}, with the threshold located\footnote{\cite{mossel-consist} allows for a slightly more general model where $a$ and $b$ are $\Theta(1)$ and gives the behaviour at the threshold. Note that at the threshold, one has to distinguish the case of $b=0$ and $b>0$ (assuming $a>b$), since for $b=0$ the clusters are not connected whp and it is not possible to recover the clusters with a vanishing error probability.} at $\sqrt{a}-\sqrt{b} = \sqrt{2}$ when $a=pn/ \ln(n)$ and $b=qn/ \ln(n)$. Efficient algorithms were also obtained in these papers. Hence, weak and exact recovery are solved in the symmetric two-community SBM.

One should also mention a line of work on another community detection model called the Censored Block Model (CBM), studied in \cite{random,abbs}. This model and its variants were also studied in \cite{abbetoc,Huang_Guibas_Graphics,Chen_Huang_Guibas_Graphics,Chen_Goldsmith_ISIT2014,abbs-isit,rough}. A SDP relaxation as in \cite{abh} for the SBM was first proposed in \cite{abbs} for the CBM, with a performance gap having roughly a factor 2. This gap was recently closed in \cite{new-xu}. 
SDP relaxations for block models were also studied in \cite{chen-xu,levina,sbm-groth}. Note that SDP algorithms are polynomial time but far from quasi-linear time. For the CBM, recent works \cite{new-vu,florent_CBM} obtained tight bounds for weak recovery using spectral methods.  

Two recent works \cite{sbm-groth,new-vu} have also obtained bounds for partial recovery in the SBM with multiple communities, for the case of symmetric blocks or with bounds on the connectivity probabilities in terms of symmetric blocks. No phase transitions for exact or weak recovery have yet been proved for the SBM with more than two communities. 


\section{Open problems}
Several extensions would be interesting for the SBM with specified parameters, such as considering parameters that vary with $n$, in particular for the number of communities, or communities of sub-linear sizes. Part of the results obtained in this paper should extend without much difficulty to some of these cases. It would also be interesting to investigate how the complexity of algorithms scales with the number of communities.\footnote{In \cite{chen-xu} this question is studied for coarser regimes of the parameters.} It would also be important to obtain results and algorithms that do not rely on the knowledge of the model parameters. Here also, some of the techniques in this paper may extend. 

For partial recovery, it would interesting to obtain tight upper-bounds on the accuracy of the reconstruction in the general SBM, in particular for the regime of large constant degrees, to check if the bound obtained in this paper is  tight. For the symmetric case, the information-theoretic and computational thresholds for weak-recovery remain open for more than 2 communities. 

Finally, there are many interesting other models to investigate, such as the Censored Block Model \cite{random,abbs,abbs-isit,abbetoc,Huang_Guibas_Graphics,Chen_Huang_Guibas_Graphics,Chen_Goldsmith_ISIT2014,rough}, the Labelled Block Model \cite{label_marc,jiaming} and many more. It would be natural to expect that for these models as well, an information-measure \`a la CH-divergence obtained in this paper determines the recovery threshold. Obtaining such variants would provide major insight towards a theory for community detection in general network models.

\section{Partial Recovery}\label{partial-sec}
\subsection{Formal results}

\begin{theorem}\label{thm1}
Given any $k\in \mathbb{Z}$, $p\in (0,1)^k$ with $|p|=1$, and symmetric matrix $Q$ with no two rows equal, let $\lambda$ be the largest eigenvalue of $PQ$, and $\lambda'$ be the eigenvalue of $PQ$ with the smallest nonzero magnitude. For any \[3\max[\ln(\lambda^2/(\lambda')^2)/\ln((\lambda')^2/\lambda),\ln(2\lambda^2/(\lambda')^2)/\ln(2\lambda^3/(\lambda')^2)]<\epsilon<1,\]
\[ 0<x<\min\left(\frac{\lambda k}{|\lambda'|\min p_i},-(\min p_i)^{-1/2}+\sqrt{1/\min p_i+\min|P_W(e_i-e_j)\cdot P^{-1}P_W(e_i-e_j)|/13}\right)\]
where $P_W(e_i-e_j)$ is the projection of $e_i-e_j$ on to $W$, and the last $\min$ is taken over all communities $i,j$ and eigenspaces $W$ of $PQ$ such that $P_W(e_i)\ne P_W(e_j)$, and
\[2ke^{-\frac{x^2(\lambda')^2\min p_i}{16\lambda k^{3/2}((\min p_i)^{-1/2}+x)}}/(1-e^{-\frac{x^2(\lambda')^2\min p_i}{16\lambda k^{3/2}((\min p_i)^{-1/2}+x)}\cdot((\frac{(\lambda')^4}{4\lambda^3})-1)})<y<\frac{\min p_i}{4\ln(4k)}\]
(which may not exist\footnote{These parameter will exist in our applications of the theorem.}) there exists an algorithm that detects communities in graphs drawn from $\gss(n,p,Q)$ with accuracy at least $1-2y$ at least $1-o(1)$ of the time and runs in $O(n^{1+\epsilon})$ time.
\end{theorem}

We refer to Section \ref{main-res} for a less technical statement of the theorem. We next provide further details on the example provided in Section \ref{main-res}, in particular concerning the constants.  
\begin{corollary}\label{full-coro}
Consider the $k$-block symmetric case. In other words, $p_i=\frac{1}{k}$ for all $i$, and $Q_{i,j}$ is $\alpha$ if $i=j$ and $\beta$ otherwise. The vector whose entries are all $1$s is an eigenvector of $PQ$ with eigenvalue $\frac{\alpha+(k-1)\beta}{k}$, and every vector whose entries add up to $0$ is an eigenvector of $PQ$ with eigenvalue $\frac{\alpha-\beta}{k}$. So, $\lambda=\frac{\alpha+(k-1)\beta}{k}$ and $\lambda'=\frac{\alpha-\beta}{k}$
and 
\begin{align}
\rho > 4 \quad \Leftrightarrow \quad \frac{(a-b)^2}{4k(a+(k-1)b)} >4,
\end{align}
which is the signal-to-noise ratio appearing in the conjectures on the detection threshold for multiple blocks \cite{decelle,mossel-sbm}. We then have 
\begin{align}
&\min_{i,j,W\in \text{eigenspaces of PQ};P_W(e_i)\ne P_W(e_j)}|P_W(e_i-e_j)\cdot P^{-1}P_W(e_i-e_j)|/13\\
&=|(e_1-e_2)\cdot P^{-1}(e_1-e_2)|/13=2k/13.
\end{align}
So, \[0<x<\sqrt{15k/13}-\sqrt{k}\] and as long as $k(\alpha+(k-1)\beta)^7<(\alpha-\beta)^8$ and $4k(\alpha+(k-1)\beta)^3<(\alpha-\beta)^4$, there exists an algorithms that detects communities, and the accuracy is $$1-O(e^{-c(\alpha-\beta)^2/(k(\alpha+(k-1)\beta))})$$ for sufficiently large $((\alpha-\beta)^2/(k(\alpha+(k-1)\beta)))$, where $c=x^2/16k^{7/2}(x+\sqrt{k})$. 
\end{corollary}

Considering the way $\epsilon$, $x$, and $y$ scale when $Q$ is multiplied by a scalar yields the following corollary.

\begin{corollary}\label{partial-delta}
For any $k\in \mathbb{Z}$, $p\in (0,1)^k$ with $|p|=1$, and symmetric  matrix $Q$ with no two rows equal, there exist $\epsilon(\delta)=O(1/\ln(\delta))$ and constant $c_1>0$ such that for all sufficiently large $\delta$, {\tt Sphere-comparison} detects communities in graphs drawn from $\gss(n,p,\delta Q)$ with accuracy at least $1-O_\delta(e^{-c_1\delta})$ in $O_n(n^{1+\epsilon(\delta)})$ time for all sufficiently large $n$.
\end{corollary}

\begin{corollary}
For any $k\in \mathbb{Z}$, $p\in (0,1)^k$ with $|p|=1$, symmetric  matrix $Q$ with no two rows equal, $b>0$, and $1>\epsilon>0$, there exists $c>0$, such that {\tt Sphere-comparison} detects communities in graphs drawn from $\gss(n,p,cQ)$ with accuracy at least $1-b$ in $O(n^{1+\epsilon})$ time for sufficiently large $n$.
\end{corollary}

If instead of having constant average degree, one has an average degree which increases as $n$ increases, one can slowly reduce $b$ and $\epsilon$ as $n$ increases, leading to the following corollary.

\begin{corollary}
For any $k\in \mathbb{Z}$, $p\in [0,1]^k$ with $|p|=1$, symmetric  matrix $Q$ with no two rows equal, and $c(n)$ such that $c=\omega(1)$, {\tt Sphere-comparison} detects the communities with accuracy $1-o(1)$ in $\gss(n,p,c(n)Q)$ and runs in $o(n^{1+\epsilon})$ time for all $\epsilon>0$.
\end{corollary}

These corollaries are important as they show that if the entries of the connectivity matrix $Q$ are amplified by a coefficient growing with $n$, almost exact recovery is achieved by ({\tt Sphere-comparison}). 

\subsection{Proof of Theorem \ref{thm1}}
Proving Theorem \ref{thm1} will require establishing some terminology. First, let $\lambda_1,...,\lambda_h$ be the distinct eigenvalues of $PQ$, ordered so that $|\lambda_1|\ge|\lambda_2|\ge...\ge|\lambda_h|\ge 0$. Also define $\eta$ so that $\eta=h$ if $\lambda_h\ne 0$ and $\eta=h-1$ if $\lambda_h=0$. In addition to this, let $d$ be the largest sum of a column of $PQ$.
\begin{definition}
For any graph $G$ drawn from $\gss(n,p,Q)$ and any set of vertices in $G$, $V$, let $\overrightarrow{V}$ be the vector such that $\overrightarrow{V}_i$ is the number of vertices in $V$ that are in community $i$. Define $w_1(V)$, $w_2(V)$, ..., $w_h(V)$ such that $\overrightarrow{V}=\sum w_i(V)$ and $w_i(V)$ is an eigenvector of $PQ$ with eigenvalue $\lambda_i$ for each $i$.
\end{definition}

$w_1(V), ...,w_h(V)$ are well defined because $\mathbb{R}^k$ is the direct sum of $PQ$'s eigenspaces. The key intuition behind their importance is that if $V'$ is the set of vertices adjacent to vertices in $V$ then $\overrightarrow{V'}\approx PQ\overrightarrow{V}$, so $w_i(V')\approx PQ\cdot w_i(V)=\lambda_iw_i(V)$.

\begin{definition}
For any vertex $v$, let $N_r(v)$ be the set of all vertices with shortest path to $v$ of length $r$. If there are multiple graphs that $v$ could be considered a vertex in, let $N_{r[G']}(v)$ be the set of all vertices with shortest paths in $G'$ to $v$ of length $r$.
\end{definition}
We also typically refer to $\overrightarrow{N_{r[G']}(v)}$ as simply $N_{r[G']}(v)$, as the context will make it clear whether the expression refers to a set or vector. 

\begin{definition}
A vertex $v$ of a graph drawn from $\gss(n,p,Q)$ is $(R,x)$-good if for all $0\le r<R$ and $w\in \mathbb{R}^k$ with $w\cdot Pw=1$ \[|w\cdot N_{r+1}(v)-w\cdot PQN_r(v)|\le \frac{x\lambda_{\eta}}{2}\left(\frac{\lambda_{\eta}^2}{2\lambda_1}\right)^{r}\] and $(R,x)$-bad otherwise.
\end{definition}

Note that since any such $w$ can be written as a linear combination of the $e_i$, $v$ is $(R,x)$-good if $|e_i\cdot N_{r+1}(v)-e_i\cdot PQN_r(v)|\le \frac{x\lambda_{\eta}}{2}\left(\frac{\lambda_{\eta}^2}{2\lambda_1}\right)^{r}\sqrt{p_i/k}$ for all $1\le i\le k$ and $0\le r<R$.

\begin{lemma}
If $v$ is a $(R,x)$-good vertex of a graph drawn from $\gss(n,p,Q)$, then for every $0\le r\le R$, $|N_r(v)|\le \lambda_1^r\sqrt{k}((\min p_i)^{-1/2}+x)$.
\end{lemma}

\begin{proof}
First, note that for any eigenvector of $PQ$, $w$, and $r<R$, \[|(P^{-1}w)\cdot N_{r+1}(v)-(P^{-1}w)\cdot PQN_r(v)|\le \frac{x\lambda_{\eta}}{2}\left(\frac{\lambda_{\eta}^2}{2\lambda_1}\right)^{r}\sqrt{w\cdot P^{-1}w}\] So, by the triangle inequality,
\begin{align*}
|(P^{-1}w)\cdot N_{r+1}(v)|&\le |(P^{-1} PQw)\cdot N_r(v)|+\frac{x\lambda_{\eta}}{2}\left(\frac{\lambda_{\eta}^2}{2\lambda_1}\right)^{r}\sqrt{w\cdot P^{-1}w}\\
&\le \lambda_1|(P^{-1}w)\cdot N_r(v)|+x\left(\frac{\lambda_1}{2}\right)^{r+1}\sqrt{w\cdot P^{-1}w}
\end{align*}

Thus, for any $r\le R$, it must be the case that 
\begin{align*}
|(P^{-1}w)\cdot N_r(v)|&\le \lambda_1^r|(P^{-1}w)\cdot N_0(v)|+\sum_{r'=1}^{r} \lambda_1^{r-r'}\cdot x\left(\frac{\lambda_1}{2}\right)^{r'}\sqrt{w\cdot P^{-1}w}\\
&\le \lambda_1^r\left (|w_{\sigma_v}/p_{\sigma_v}|+x\sqrt{w\cdot P^{-1}w}\right)
\end{align*}

Now, define $w_1$,..., $w_h$ such that $PQw_i=\lambda_iw_i$ for each $i$ and $p=\sum_{i=1}^h w_i$. For any $i,j$, 
\begin{align*}
\lambda_iw_i\cdot P^{-1}w_j&=(PQw_i)\cdot P^{-1}w_j\\
&=w_i\cdot P^{-1}PQw_j\\
&=\lambda_jw_i\cdot P^{-1}w_j
\end{align*}
If $i\ne j$, then $\lambda_i\ne \lambda_j$, so this implies that $w_i\cdot P^{-1} w_j=0$. It follows from this that
\begin{align*}
\sum_i w_i\cdot P^{-1} w_i&=\sum_{i,j} w_i\cdot P^{-1}w_j\\
&= \left(\sum_i w_i\right)\cdot P^{-1}\left(\sum_j w_j\right)\\
&=p\cdot P^{-1} p=1
\end{align*}

Also, for any $i$, it is the case that $|(w_i)_{\sigma_v}/p_{\sigma_v}|\le \sqrt{(w_i)_{\sigma_v}\cdot p^{-1}_{\sigma_v}\cdot (w_i)_{\sigma_v}}/\sqrt{p_{\sigma_v}}\le (\min p_i)^{-1/2}\sqrt{w_i\cdot P^{-1}w_i}$

Therefore, for any $r\le R$, we have that
\begin{align*}
|N_r(v)|&=|(P^{-1}p)\cdot N_r(v)|\\
&\le \sum_i |(P^{-1} w_i)\cdot N_r(v)|\\
&\le \lambda_1^r\sum_i |(w_i)_{\sigma_v}/p_{\sigma_v}|+\lambda_1^r x\sum_i \sqrt{w_i\cdot P^{-1}w_i}\\
&\le \lambda_1^r\sqrt{k}((\min p_i)^{-1/2}+x)
\end{align*}

\end{proof}

We will prove that for parameters satisfying the correct criteria, most vertices are good, but first we will need the following concentration result (see for example \cite{max-igal} page 19).
\begin{theorem}
Let $X_1,...,X_n$ be a sequence of independent random variables and $d,\sigma\in\mathbb{R}$ such that for all $i$, $|X_i-E[X_i]|<d$ with probability $1$ and $Var[X_i]\le\sigma^2$. Then for every $\alpha>0$, \[P\left(|\sum_{i=1}^n X_i-E\left[\sum_{i=1}^n X_i\right]|\ge \alpha n\right)\le 2e^{-nD(\frac{\delta+\gamma}{1+\gamma}||\frac{\gamma}{1+\gamma})}\] where $\delta=\alpha/d$, $\gamma=\sigma^2/d^2$, and $D(p||q)= p\ln(p/q)+(1-p)\ln((1-p)/(1-q))$.
\end{theorem}

Note that for any vertex $v\in G$, $r\in\mathbb{Z}$, and $1\le i\le k$, 
\[e_i\cdot N_r(v)=\sum_{v'\in G} I_{(v'\in N_r(v))} e_i\cdot \overrightarrow{\{v'\}}\] where $I_{(v'\in N_r(v))}$ is $1$ if $v'$ is in $N_r(v)$ and $0$ otherwise. Note that $$|e_i\cdot \overrightarrow{\{v'\}}|\le 1$$ for all $v'$, and $$E[ (I_{(v'\in N_r(v))} e_i\cdot \overrightarrow{\{v'\}})^2]\le \frac{d|N_{r-1}(v)|}{n}$$ because $I_{(v'\in N_r(v))}$ is nonzero with probability at most $d|N_{r-1}(v)|/n$. A vertex in community $\sigma$ that is not in $N_{r'}(v)$ for $r'<r$ is in $N_r(v)$ with a probability of approximately $1-e^{-e_\sigma QN_{r-1}(v)/n}$, and there are approximately $p_\sigma n-O(|\cup_{r'<r}N_{r'}(v)|)$ such vertices, so the expected value of $e_i\cdot N_r(v)$ differs from $e_i\cdot PQN_{r-1}(v)$ by a term which is at most proportional to $|N_{r-1}(v)|\cdot \sum_{i=0}^{r-1} |N_i(v)|/n$.

Theorem $2$ can be applied to this formula in order to bound the probability that a vertex will be bad, but first we need the following lemma.

\begin{lemma}
For any $0<\delta,\gamma<1$, $D(\frac{\delta+\gamma}{1+\gamma}||\frac{\gamma}{1+\gamma})>\frac{\delta^2(\gamma-\delta)}{2\gamma^2(1+\gamma)}$.
\end{lemma}

\begin{proof}
First, note that if $0<x<1$ then $\ln(1+x)=\sum_{i=1}^\infty \frac{(-1)^{i+1}x^i}{i}>x-\frac{x^2}{2}=x\cdot\frac{2-x}{2}$. Similarly, if $x<0$ then $\ln(1+x)=-\ln(1/(1+x))=-\ln(1-x/(1+x))>\frac{x}{1+x}$. So,
\begin{align*}
D\left(\frac{\delta+\gamma}{1+\gamma}||\frac{\gamma}{1+\gamma}\right)&=\frac{\delta+\gamma}{1+\gamma}\ln\left(\frac{\delta+\gamma}{\gamma}\right)+\frac{1-\delta}{1+\gamma}\ln(1-\delta)\\
&>\frac{\delta+\gamma}{1+\gamma}\cdot\frac{\delta}{\gamma}\cdot\frac{2\gamma-\delta}{2\gamma}+\frac{1-\delta}{1+\gamma}\cdot\frac{-\delta}{1-\delta}\\
&=\frac{\delta}{1+\gamma}\left[ \frac{(\delta+\gamma)(2\gamma-\delta)}{2\gamma^2}-1\right]\\
&=\frac{\delta}{1+\gamma}\cdot\frac{\gamma\delta-\delta^2}{2\gamma^2}\\
&=\frac{\delta^2(\gamma-\delta)}{2\gamma^2(1+\gamma)}.
\end{align*}
\end{proof}

\begin{lemma}
Let $k\in \mathbb{Z}$, $p\in (0,1)^k$ with $|p|=1$, $Q$ be a symmetric matrix such that $\lambda_{\eta}^4>4\lambda_1^3$, and $0<x<\frac{\lambda_1k}{\lambda_{\eta}\min p_i}$. Then there exists \[ y<2ke^{-\frac{x^2\lambda_{\eta}^2\min p_i}{16\lambda_1 k^{3/2}((\min p_i)^{-1/2}+x)}}/\left(1-e^{-\frac{x^2\lambda_{\eta}^2\min p_i}{16\lambda_1k^{3/2}((\min p_i)^{-1/2}+x)}\cdot((\frac{\lambda_{\eta}^4}{4\lambda_1^3})-1)}\right)\] and $R(n)= \omega(1)$  such that at least $1-y$ of the vertices of a graph drawn from $\gss(n,p,Q)$ are $(R(n),x)$-good with probability $1-o(1)$.
\end{lemma}

\begin{proof}
First, consider a constant R. Now, define $w_1$,..., $w_h$ such that $PQw_i=\lambda_iw_i$ for each $i$ and $p=\sum_{i=1}^h w_i$. If $v$ is $(r,x)$-good then by the same logic used in the proof of lemma $1$, each vertex of $G$ is in $N_{r+1}(v)$ with probability at most 
\begin{align*}
p\cdot QN_r(v)/n&=\frac{1}{n}\sum_i (P^{-1} w_i)\cdot PQN_r(v)\\
&\le \frac{1}{n}\lambda_1^{r+1}\sum_i |(w_i)_{\sigma_v}/p_{\sigma_v}|+\frac{1}{n}\lambda_1^{r+1}x\sum_i\sqrt{w_i\cdot P^{-1}w_i}\\
&\le \frac{1}{n}\lambda_1^{r+1}\sqrt{k}((\min p_i)^{-1/2}+x)
\end{align*}

 Recall that a vertex $v$ is $(R,x)$-good if (but not only if) \[|e_i\cdot N_{r+1}(v)-e_i\cdot PQN_r(v)|\le \frac{x\lambda_{\eta}}{2}\left(\frac{\lambda_{\eta}^2}{2\lambda_1}\right)^{r}\sqrt{p_i/k}\] for all $0\le r<R$ and $1\le i\le k$. This condition holds for $i$ and $r$ with probability at least $1-2e^{-nD(\frac{\delta+\gamma}{1+\gamma}||\frac{\gamma}{1+\gamma})}$, where $\delta\sim \frac{x\lambda_{\eta}}{2}(\frac{\lambda_{\eta}^2}{2\lambda_1})^{r}(\sqrt{p_i/kn^2})$ and $\gamma\sim \lambda_1^{r+1}\sqrt{k}((\min p_i)^{-1/2}+x)/n=\Omega(d|N_{r}(v)|/n)$ That means that in the limit as $n\to\infty$, $v$ is $(R,x)$-bad with probability at most 
\begin{align*}
\sum_{r=0}^{R-1}\sum_{i=1}^k 2e^{-nD(\frac{\delta_r+\gamma_r}{1+\gamma_r}||\frac{\gamma_r}{1+\gamma_r})}&\le \sum_{r=0}^{R-1} 2ke^{-n\frac{\delta^2(\gamma-\delta)}{2\gamma^2(1+\gamma)}}\\
&\le \sum_{r=0}^{R-1} 2ke^{-n\frac{\delta^2(\gamma/2+(\gamma/2-\delta))}{2\gamma^2}}\\
&\le \sum_{r=0}^{R-1} 2ke^{-n\frac{\delta^2}{4\gamma}}\\
&\le \sum_{r=0}^{R-1} 2ke^{-\frac{\frac{x^2\lambda_{\eta}^2}{4}\left(\frac{\lambda_{\eta}^2}{2\lambda_1}\right)^{2r}\min p_i/k}{4\lambda_1^{r+1}\sqrt{k}((\min p_i)^{-1/2}+x)}}\\
&\le \sum_{r=0}^\infty 2ke^{-\frac{x^2\lambda_{\eta}^2\min p_i}{16k\lambda_1\sqrt{k}((\min p_i)^{-1/2}+x)}\left(\frac{\lambda_{\eta}^4}{4\lambda_1^3}\right)^r}\\
&< \sum_{r=0}^\infty 2ke^{-\frac{x^2\lambda_{\eta}^2\min p_i}{16k^{3/2}\lambda_1((\min p_i)^{-1/2}+x)}\left(1+\left(\left(\frac{\lambda_{\eta}^4}{4\lambda_1^3}\right)-1\right)r\right)}\\
&=2ke^{-\frac{x^2\lambda_{\eta}^2\min p_i}{16\lambda_1k^{3/2}((\min p_i)^{-1/2}+x)}}/\left(1-e^{-\frac{x^2\lambda_{\eta}^2\min p_i}{16\lambda_1k^{3/2}((\min p_i)^{-1/2}+x)}\cdot((\frac{\lambda_{\eta}^4}{4\lambda_1^3})-1)}\right)
\end{align*}

Given random $v$ and $v'$, if $v'$ is $(R,x)$-good then there are at most $\sum_{r=0}^R \lambda_1^r\sqrt{k}((\min p_i)^{-1/2}+x)$ vertices in $\cup_{r=0}^R N_r(v')$. Note that $\cup_{r=0}^R N_r(v)$ is disjoint from any set of $\sum_{r=0}^R \lambda_1^r\sqrt{k}((\min p_i)^{-1/2}+x)$ vertices that were chosen independently of $v$ with probability $1-O(1/n)$, so \[|P[v \text{ is } (R,x)-\text{good}]-P[v \text{ is } (R,x)-\text{good}|v' \text{ is } (R,x)-\text{good}]|=O(1/n).\] That means that for any $$y<\sum_{r=0}^\infty 2ke^{-\frac{x^2\lambda_{\eta}^2\min p_i}{16k^{3/2}\lambda_1((\min p_i)^{-1/2}+x)}\left(1+\left(\left(\frac{\lambda_{\eta}^4}{4\lambda_1^3}\right)-1\right)r\right)},$$ at least $(1-y)n$ of the vertices in a graph drawn from $\gss(n,p,Q)$ are $(R,x)$-good with probability $1-o(1)$.

So, for every $r$ there exists $N_r$ such that for all $n>N_r$, at least $(1-y)n$ of the vertices of a graph drawn from $G(p, Q, n)$ are $(r,x)$-good with probability at least $1-2^{-r}$. Now, let $R(n)=\sup \{r: n>N_r\}$. It is clear that $\lim_{n\to\infty} R(n)=\infty$, and for any $n$, at least $(1-y)n$ of the vertices of a graph drawn from $G(p, Q, n)$ are $(R(n),x)$ good with probability at least $1-2^{-R(n)}=1-o(1)$.
\end{proof}

\begin{lemma}
Let $k\in \mathbb{Z}$, $p\in (0,1)^k$ with $|p|=1$, $Q$ be a symmetric matrix such that $\lambda_{\eta}^4>4\lambda_1^3$, $R(n)=\omega(1)$, and $\epsilon>0$ such that $(2\lambda_1^3/\lambda_{\eta}^2)^{1-\epsilon/3}<\lambda_1$. A vertex of a graph drawn from $G(p, Q, n)$ is $(R(n),x)$-good but $(\frac{1-\epsilon/3}{\ln\lambda_1}\ln n,x)$-bad with probability $o(1)$.
\end{lemma}

\begin{proof}
for any $r<\frac{1-\epsilon/3}{\ln\lambda_1}\ln n$, if $v$ is $(r,x)$-good then $$|\cup_{i=0}^r N_r(v)|\le\sum_{i=0}^r  \lambda_1^r\sqrt{k}((\min p_i)^{-1/2}+x)<\lambda_1^{r+1}\sqrt{k}((\min p_i)^{-1/2}+x)/(\lambda_1-1).$$ By assumption, 
\begin{align*}
&|N_{r-1}(v)|\cdot|\cup_{i=0}^r N_i(v)|/\left(\frac{x|\lambda_{\eta}|}{2}\left(\frac{\lambda_{\eta}^2}{2\lambda_1}\right)^{r}\sqrt{p_j/k}\right)\\
&\le (2\lambda_1^3/\lambda_{\eta}^2)^{r}\cdot 2((\min p_i)^{-1/2}+x)^2\sqrt{k^3/p_j}/(x(\lambda_1-1)|\lambda_{\eta}|)\\ 
&\le (2\lambda_1^3/\lambda_{\eta}^2)^{\frac{1-\epsilon/3}{\ln\lambda_1}\ln n}\cdot 2((\min p_i)^{-1/2}+x)^2\sqrt{k^3/p_j}/(x(\lambda_1-1)|\lambda_{\eta}|)\\
&=o(n)
\end{align*}
So, if $n$ is sufficiently large, then the expected value of $e_i\cdot N_{r+1}(v)$ differs from $e_i\cdot PQN_r(v)$ by less than $\frac{1}{2}\cdot\frac{x|\lambda_{\eta}|}{2}(\frac{\lambda_{\eta}^2}{2\lambda_1})^{r}\sqrt{p_i/k}$ for all $r<\frac{1-\epsilon/3}{\ln\lambda_1}\ln n$. For such an $n$, a $(R(n),x)$ good vertex is also $(\frac{1-\epsilon/3}{\ln\lambda_1}\ln n,x)$-good if $e_i\cdot N_{r+1}$ differs from its expected value by at most $\frac{x|\lambda_{\eta}|}{4}(\frac{\lambda_{\eta}^2}{2\lambda_1})^{r}\sqrt{p_i/k}$ for all $R(n)\le r<\frac{1-\epsilon/3}{\ln\lambda_1}\ln n$ and $1\le i\le k$. Note that $$\sum_{r=0}^{\frac{1-\epsilon/3}{\ln\lambda_1}\ln n} \lambda_1^r\sqrt{k}((\min p_i)^{-1/2}+x)=o(n),$$ so for a given $i$ and $r$ this holds with probability at least   $1-2e^{-nD(\frac{\delta+\gamma}{1+\gamma}||\frac{\gamma}{1+\gamma})}$, where $\delta\sim \frac{x|\lambda_{\eta}|}{4}(\frac{\lambda_{\eta}^2}{2\lambda_1})^{r}(\sqrt{p_i/kn^2})$ and $\gamma\sim \lambda_1^{r+1}\sqrt{k}((\min p_i)^{-1/2}+x)/n$. 
\begin{align*}
D\left(\frac{\delta+\gamma}{1+\gamma}||\frac{\gamma}{1+\gamma}\right)&>\frac{\delta^2[\gamma-\delta]}{2\gamma^2(1+\gamma)}\\
&\sim \frac{\delta^2\gamma}{2\gamma^2}\\
&=\frac{\delta^2}{2\gamma}\\
&\sim \frac{x^2\lambda_{\eta}^2p_i}{32\lambda_1k^{3/2}((\min p_i)^{-1/2}+x)n}\left(\frac{\lambda_{\eta}^4}{4\lambda_1^3}\right)^{r}\\
&\ge \frac{x^2\lambda_{\eta}^2p_i}{32\lambda_1k^{3/2}((\min p_i)^{-1/2}+x)n}\left(1+r\left(\frac{\lambda_{\eta}^4}{4\lambda_1^3}-1\right)\right)
\end{align*}
So, there exist $N$, $a$, and $b>0$ such that if $n>N$ and $r<\frac{1-\epsilon/3}{\ln\lambda_1}\ln n$, then $D(\frac{\delta+\gamma}{1+\gamma}||\frac{\gamma}{1+\gamma})>(a+br)/n$. So, the probability that $v$ is $(R(n),x)$-good but $(\frac{1-\epsilon/3}{\ln\lambda_1}\ln n,x)$-bad is at most 
\begin{align*}
2k\sum_{r=R(n)}^{\frac{1-\epsilon/3}{\ln\lambda_1}\ln n} e^{-nD(\frac{\delta+\gamma}{1+\gamma}||\frac{\gamma}{1+\gamma})}&\le2k\sum_{r=R(n)}^{\frac{1-\epsilon/3}{\ln\lambda_1}\ln n} e^{-a-br}\\
&<2ke^{-a-bR(n)}/(1-e^{-b})\\
&=o(1)
\end{align*}
\end{proof}

\begin{definition}
For any vertices $v, v'\in G$, $r,r'\in \mathbb{Z}$, and subset of $G$'s edges $E$, let $N_{r,r'[E]}(v\cdot v')$ be the number of pairs of vertices $(v_1,v_2)$ such that $v_1\in N_{r[G\backslash E]}(v)$, $v_2\in N_{r'[G\backslash E]}(v')$, and $(v_1,v_2)\in E$.
\end{definition}

Note that if $N_{r[G\backslash E]}(v)$ and $N_{r'[G\backslash E]}(v')$ have already been computed, $N_{r,r'[E]}(v\cdot v')$ can be computed by means of the following algorithm, where $E[v]=\{v':(v,v')\in E\}$
\begin{algorithm}
compute\_$N_{r,r'[E]}(v\cdot v')$:

for $v_1\in N_{r'[G\backslash E]}(v')$:

$\phantom{xxx}$ for $v_2\in E[v_1]:$

$\phantom{xxxxxx}$ if $v_2\in N_{r[G\backslash E]}(v):$

$\phantom{xxxxxxxxx}$ count=count+1

return count
\end{algorithm}

Note that this runs in $O((d+1)|N_{r'[G\backslash E]}(v')|)$ average time. The plan is to independently put each edge in $G$ in $E$ with probability $c$. Then the probability distribution of $G\backslash E$ will be $\gss(n,p,(1-c)Q)$, so $N_{r[G\backslash E]}(v)\approx ((1-c)PQ)^re_{\sigma_v}$ and $N_{r'[G\backslash E]}(v')\approx ((1-c)PQ)^{r'}e_{\sigma_{v'}}$. So, it will hopefully be the case that \[N_{r,r'[E]}(v\cdot v')\approx ((1-c)PQ)^re_{\sigma_v}\cdot cQ((1-c)PQ)^{r'}e_{\sigma_{v'}}/n= c(1-c)^{r+r'} e_{\sigma_v}\cdot Q(PQ)^{r+r'}e_{\sigma_{v'}}/n.\] More rigorously, we have that:

\begin{lemma}
Choose $p$, $Q$, $G$ drawn from $\gss(n,p,Q)$, $E$ randomly selected from $G$'s edges such that each of $G$'s edges is independently assigned to $E$ with probability $c$, and $v,v'\in G$ chosen independently from $G$'s vertices. Then with probability $1-o(1)$, \[|N_{r,r'[E]}(v\cdot v')-N_{r[G\backslash E]}(v)\cdot cQN_{r'[G\backslash E]}(v')/n|<(1+\sqrt{|N_{r[G\backslash E]}(v)|\cdot |N_{r'[G\backslash E]}(v')|/n})\log n\]
\end{lemma}

\begin{proof}
Roughly speaking, for each $v_1\in N_{r[G\backslash E]}(v)$ and $v_2\in N_{r'[G\backslash E]}(v')$, $(v_1,v_2)\in E$ with probability $cQ_{\sigma_{v_1},\sigma_{v_2}}/n$. This is complicated by the facts that $(v_1,v_1)$ is never in $E$ and no edge is in $G\backslash E$ and $E$. However, this changes the expected value of $N_{r,r'[E]}(v\cdot v')$ given $G\backslash E$ by at most a constant unless $G$ has more than double its expected number of edges, something that happens with probability $o(1)$. Furthermore, whether $(v_1,v_2)$ is in $E$ is independent of whether $(v_1',v_2')$ is in $E$ unless $(v_1',v_2')=(v_1,v_2)$ or $(v_1',v_2')=(v_2,v_1)$. So, the variance of  $N_{r,r'[E]}(v\cdot v')$ is proportional to its expected value, which is $$O(|N_{r[G\backslash E]}(v)|\cdot |N_{r'[G\backslash E]}(v')|/n).$$ $N_{r,r'[E]}(v\cdot v')$ is within $\log n$ standard deviations of its expected value with probability $1-o(1)$, which completes the proof.
\end{proof}

Note that if $\overrightarrow{v}$ is an eigenvector of $(1-c)PQ$, $\sqrt{P}Q\overrightarrow{v}$ is an eigenvector of the symmetric matrix $(1-c)\sqrt{P}Q\sqrt{P}$. So, since eigenvectors of a symmetric matrix with different eigenvalues are orthogonal, we have \[N_{r[G\backslash E]}(v)\cdot cQN_{r'[G\backslash E]}(v')/n=\frac{c}{n} \sum_i w_i(N_{r[G\backslash E]}(v))\cdot Qw_i(N_{r'[G\backslash E]}(v'))\] 

\begin{lemma}[Decomposition Equation Lemma]
Let $x>0$, $0<c<1$ such that $(1-c)\lambda_{\eta}^2>\lambda_1$, $\epsilon>0$, $G$ drawn from $\gss(n,p,Q)$, $E$ be a subset of $G$'s edges that independently contains each edge with probability $c$, $r, r'\in\mathbb{Z}^+$ such that $r+r'\ge (1+\epsilon)\log n/\log((1-c)\lambda_{\eta}^2/\lambda_1)$ and $r\ge r'$, and $v,v'\in G$ be chosen independently of $G$'s adjacency matrix. The system of equations \[\sum_i ((1-c)\lambda_i)^{r+r'+j+1}z_i=\frac{(1-c)n}{c}N_{r+j,r'[E]}(v\cdot v') \text{ for } 0\le j<\eta\] has a unique solution. Furthermore, if $v$ is $(r+\eta,x)$-good and $v'$ is $(r'+1,x)$-good with respect to $G\backslash E$ then \[|z_i-w_i(\{v\})\cdot P^{-1}w_i(\{v'\})|<2x(\min p_j)^{-1/2}+x^2+o(1)\] for all $i$ with probability $1-o(1)$.
\end{lemma}

\begin{proof}
First, note that by $(r+\eta,x)$-goodness of $v$, \[|w\cdot w_i(N_{r+j[G\backslash E]}(v))-\lambda_i^j w\cdot w_i(N_{r[G\backslash E]}(v))|<x(d(1-c))^{j-1}(1-c)|\lambda_{\eta}|\left(\frac{(1-c)\lambda_{\eta}^2}{2\lambda_1}\right)^{r}\] whenever $w\cdot Pw=1$. So, with probability $1-o(1)$, 
\begin{align*}
&|\sum_i ((1-c)\lambda_i)^{j}w_i(N_{r[G\backslash E]}(v))\cdot Qw_i(N_{r'[G\backslash E]}(v'))-\frac{n}{c}N_{r+j,r'[E]}(v\cdot v')|\\
&<\left(1+\sqrt{(\max_i 1/p_i+x)^2(1-c)^{r+r'+j}\lambda_1^{r+r'+j}/n}\right)\frac{n}{c}\log n\\
&\phantom{xxxxxx}+x|\lambda_{\eta}|(1-c)^{r+r'+j}d^{j-1}\left(\frac{\lambda_{\eta}^2}{2\lambda_1}\right)^{r}\lambda_1^{r'}(\max_i 1/p_i+x)\\
&\le \left(n+((1-c)\lambda_{\eta})^{r+r'+j}\sqrt{(\max_i 1/p_i+x)^2(\lambda_1/((1-c)\lambda_{\eta}^2))^{r+r'+j}\cdot n}\right)\frac{\log n}{c}\\
&\phantom{xxxxxx}+x|\lambda_{\eta}|(1-c)^{r+r'+j}d^{j-1}\left(\frac{\lambda_{\eta}^2}{2}\right)^{(r+r')/2}(\max_i 1/p_i+x)\\
& \le \left(((1-c)\lambda_{\eta})^{\log n/\log((1-c)\lambda_{\eta})}+((1-c)\lambda_{\eta})^{r+r'+j}\sqrt{(\max_i 1/p_i+x)^2n^{-1-\epsilon}\cdot n}\right)\frac{\log n}{c}\\
&\phantom{xxxxxx}+x|\lambda_{\eta}|(1-c)^{r+r'+j}d^{j-1}\left(\frac{\lambda_{\eta}^2}{2}\right)^{(r+r')/2}(\max_i 1/p_i+x)\\
&\le \left(((1-c)\lambda_{\eta})^{(r+r')/(1+\epsilon)}+((1-c)\lambda_{\eta})^{r+r'+j}\sqrt{(\max_i 1/p_i+x)^2n^{-\epsilon}}\right)\frac{\log n}{c}\\
&\phantom{xxxxxx}+x|\lambda_{\eta}|(1-c)^{r+r'+j}d^{j-1}\left(\frac{\lambda_{\eta}^2}{2}\right)^{(r+r')/2}(\max_i 1/p_i+x)\\
&=o(((1-c)\lambda_{\eta})^{r+r'})
\end{align*}

Now, let $M$ be the matrix such that $M_{i,j}=((1-c)\lambda_i)^j$. This matrix is invertible because the $\lambda_i$ are distinct, so the system of equations has a unique solution. Furthermore, for fixed values of $c$ and $i$, $((1-c)\lambda_i)^{r+r'}\lambda_i z_i-w_i(N_{r[G\backslash E]}(v))\cdot Qw_i(N_{r'[G\backslash E]}(v'))$ is a fixed linear combination of these error terms. So,
 \begin{align*}
&|z_i-w_i(\{v\})\cdot P^{-1}w_i(\{v'\})|\\
&\le |z_i-(1-c)((1-c)\lambda_i)^{-r-r'-1}w_i(N_{r[G\backslash E]}(v))\cdot Qw_i(N_{r'[G\backslash E]}(v')|\\
&\phantom{xxxxxx}+(1-c)\cdot|((1-c)\lambda_i)^{-r-r'-1}w_i(N_{r[G\backslash E]}(v))\cdot Q w_i(N_{r'[G\backslash E]}(v')\\
&\phantom{\phantom{xxxxxx}+(1-c)\cdot|}-((1-c)\lambda_i)^{-r'-1}w_i(\{v\})\cdot Qw_i(N_{r'[G\backslash E]}(v')|\\
&\phantom{xxxxxx}+|(1-c)\cdot((1-c)\lambda_i)^{-r'-1}w_i(\{v\})\cdot Qw_i(N_{r'[G\backslash E]}(v')-w_i(\{v\})\cdot P^{-1}w_i(\{v'\})|\\
&\le |((1-c)\lambda_i)^{-r'}w_i(N_{r'[G\backslash E]}(v'))\cdot P^{-1}[((1-c)\lambda_i)^{-r}w_i(N_{r[G\backslash E]}(v))-w_i(\{v\})]|\\
&\phantom{xxxxxx}+|w_i(\{v\})\cdot P^{-1}[((1-c)\lambda_i)^{-r'}w_i(N_{r'[G\backslash E]}(v')-w_i(\{v'\}))]+o(1)
\end{align*}
By goodness of $v$ and $v'$, this is less than or equal to
\begin{align*}
& ((1-c)\lambda_i)^{-r'}\sqrt{w_i(N_{r'[G\backslash E]}(v'))\cdot P^{-1}w_i(N_{r'[G\backslash E]}(v'))}\cdot x+\sqrt{w_i(\{v\})\cdot P^{-1}w_i(\{v\})}\cdot x+o(1)\\
&\le\sqrt{w_i(\{v'\})\cdot P^{-1}w_i(\{v'\})+2w_i(\{v'\})\cdot P^{-1}[((1-c)\lambda_i)^{-r'}w_i(N_{r'[G\backslash E]}(v'))-w_i(\{v'\})] +}\\
&\overline{[((1-c)\lambda_i)^{-r'}w_i(N_{r'[G\backslash E]}(v'))-w_i(\{v'\})]\cdot P^{-1}[((1-c)\lambda_i)^{-r'}w_i(N_{r'[G\backslash E]}(v'))-w_i(\{v'\})]}\cdot x\\
&+x\sqrt{1/\min p_j}+o(1)\\
&\le \sqrt{1/\min p_j+2x/\sqrt{\min p_j}+x^2}x+x/\sqrt{\min p_j}+o(1)\\
&=(x^2+2x(\min p_j)^{-1/2})+o(1)
\end{align*}
with probability $1-o(1)$ for all $i$.
\end{proof}

For any two vertices in different communities, $v$ and $v'$, the fact that $Q$'s rows are distinct implies that $Q(\overrightarrow{\{v\}}-\overrightarrow{\{v'\}})\ne 0$. So, $w_i(\{v\})\ne w_i(\{v'\})$ for some $1\le i\le \eta$. That means that for any two vertices $v$ and $v'$,  
\begin{align*}
&(w_i(\{v\})- w_i(\{v'\}))\cdot P^{-1}(w_i(\{v\})- w_i(\{v'\}))\\
&=w_i(\{v\})\cdot P^{-1}w_i(\{v\})-2w_i(\{v\})\cdot P^{-1}w_i(\{v'\})+w_i(\{v'\})\cdot P^{-1}w_i(\{v'\})\ge 0
\end{align*}
for all $1\le i\le \eta$, with equality for all $i$ if and only if $v$ and $v'$ are in the same community. This also implies that given a vertex $v$, another vertex in the same community $v'$, and a vertex in a different community $v''$, 
\begin{align*}
&2w_i(\{v\})\cdot P^{-1}w_i(\{v'\})-w_i(\{v'\})\cdot P^{-1}w_i(\{v'\}) \\&\ge 2w_i(\{v\})\cdot P^{-1}w_i(\{v''\})-w_i(\{v''\})\cdot P^{-1}w_i(\{v''\})
\end{align*}
 for all $1\le i\le \eta$ and the inequality is strict for at least one $i$. This suggests the following algorithms for classifying vertices.

\begin{algorithm}
Vertex\_comparison\_algorithm(v,v', r,r',E,x,c):

(Assumes that $N_{r''[G\backslash E]}(v)$ and $N_{r''[G\backslash E]}(v)$ have already been computed for $r''\le r+\eta$)

\phantom{xxx} Solve the equations given in the previous lemma for $(v,v',r,r')$, $(v,v,r,r')$, and  $(v',v',r,r')$ in order to compute $z_i(v\cdot v')$, $z_i(v\cdot v)$, and $z_i(v'\cdot v')$ 

\phantom{xxx} If $\exists i: z_i(v\cdot v)-2z_i(v\cdot v')+z_i(v'\cdot v')> 5(2x(\min p_j)^{-1/2}+x^2)$ then conclude that $v$ and $v'$ are in different communities.

\phantom{xxx} Otherwise, conclude that $v$ and $v'$ are in the same community.
\end{algorithm}

\begin{lemma}
Assuming that each of $G$'s edges was independently assigned to $E$ with probability $c$, this algorithm runs in $O(((1-c)\lambda_1)^{r'})$ average time. Furthermore, if the conditions of the decomposition equation lemma are satisfied and $13(2x(\min p_j)^{-1/2}+x^2)$ is less than the minimum nonzero value of $(w_i(\{v\})- w_i(\{v'\}))\cdot P^{-1}(w_i(\{v\})- w_i(\{v'\}))$ then the algorithm returns the correct result with probability $1-o(1)$.
\end{lemma}

\begin{proof}
The slowest step of the algorithm is using $compute\_N_{r+j,r'[E]}(v\cdot v')$ in order to calculate the constant terms for the equations. This runs in an average time of $O(E[|N_{r'[G\backslash E]]}(v)|+E[|N_{r'[G\backslash E]}(v')|])= O(((1-c)\lambda_1)^{r'})$ and must be done $3\eta$ times. If the conditions of the decomposition equation lemma are satisified then with probability $1-o(1)$ the $z_i$ are all within $\frac{6}{5}(2x(\min p_j)^{-1/2}+x^2)$ of the products they seek to approximate, in which case $$z_i(v\cdot v)-2z_i(v\cdot v')+z_i(v'\cdot v')> 5(2x(\min p_j)^{-1/2}+x^2)$$ if and only if $$(w_i(\{v\})- w_i(\{v'\}))\cdot P^{-1}(w_i(\{v\})- w_i(\{v'\}))\ne 0,$$ which is true for some $i$ if and only if $v$ and $v'$ are in different communities.
\end{proof}

\begin{algorithm}
Vertex\_classification\_algorithm(v[],v', r,r',E,x,c):

(Assumes that $N_{r''[G\backslash E]}(v[\sigma])$have already been computed for $0\le \sigma<k$ and $r''\le r+\eta$, that $N_{r''[G\backslash E]}(v')$ has already been computed for all $r''\le r'$, and that $z_i(v[\sigma]\cdot v[\sigma] )$ as described in the previous algorithm have already been computed for each $i$ and $\sigma$)

\phantom{xxx} Solve the equations in the decomposition equation lemma for $(v[\sigma],v',r,r')$ in order to compute $z_i(v[\sigma]\cdot v')$ for each $\sigma$

\phantom{xxx} If there exists a unique $\sigma$ such that for all $\sigma'\ne \sigma$ and all $i$, \[z_i(v[\sigma]\cdot v[\sigma])-2z_i(v[\sigma]\cdot v')\le z_i(v[\sigma']\cdot v[\sigma'])-2z_i(v[\sigma']\cdot v') +\frac{19}{3}\cdot (2x(\min p_j)^{-1/2}+x^2)\] then conclude that $v'$ is in the same community as $v[\sigma]$.

\phantom{xxx} Otherwise, Fail.
\end{algorithm}

\begin{lemma}
Assuming that $E$ was generated properly, this algorithm runs in $O(((1-c)\lambda_1)^{r'})$ average time. Furthermore, assume that $r$, $r'$, $x$, $c$, and the graph's parameters satisfy the conditions of the decomposition equation lemma. Also, assume that $v[]$ contains exactly one vertex from each community, $v[\sigma]$ is $( r+\eta,x)$-good with respect to $G\backslash E$ for all $\sigma$, and $v'$ is $(r'+1,x)$-good with respect to $G\backslash E$. Finally, assume that $13(2x(\min p_j)^{-1/2}+x^2)$ is less than the minimum nonzero value of $(w_i(\{v\})- w_i(\{v'\}))\cdot P^{-1}(w_i(\{v\})- w_i(\{v'\}))$. Then this algorithm classifies $v'$ correctly with probability $1-o(1)$.
\end{lemma}

\begin{proof}
Again, the slowest step of the algorithm is using $compute\_N_{r+j,r'[E]}(v[\sigma]\cdot v')$ in order to calculate the constant terms for the equations. This runs in an average time of  $O((d+1)E[|N_{r'[G\backslash E]}(v')|])= O(((1-c)\lambda_1)^{r'})$ and must be done $k\eta$ times. If the conditions given above are satisifed, then each $z_i$ is within  $\frac{21}{20}(2x(\min p_j)^{-1/2}+x^2)$ of the product it seeks to approximate with probability $1-o(1)$. If this is the case, then \[z_i(v[\sigma]\cdot v[\sigma])-2z_i(v[\sigma]\cdot v')\le z_i(v[\sigma']\cdot v[\sigma])-2z_i(v[\sigma']\cdot v') +\frac{19}{3}\cdot (2x(\min p_j)^{-1/2}+x^2)\]
 iff 
\begin{align*}
&2w_i(\{v'\})\cdot P^{-1}w_i(\{v[\sigma]\})-w_i(\{v[\sigma]\})\cdot P^{-1}w_i(\{v[\sigma]\})\\
&\ge 2w_i(\{v'\})\cdot P^{-1}w_i(\{v[\sigma']\})-w_i(\{v[\sigma']\})\cdot P^{-1}w_i(\{v[\sigma']\})
\end{align*}
 This holds for all $i$ and $\sigma'$ iff $v'$ is in the same community as $v[\sigma]$, so the algorithm returns the correct result with probability $1-o(1)$.
\end{proof}

At this point, we can finally start giving algorithms for classifying a graph's vertices.
\begin{algorithm}
Unreliable\_graph\_classification\_algorithm(G,c,m,$\epsilon$,x):

\phantom{xxx} Randomly assign each edge in $G$ to $E$ independently with probability $c$.

\phantom{xxx} Randomly select $m$ vertices in $G$, $v[0],...,v[m-1]$.

\phantom{xxx} Let $r=(1-\frac{\epsilon}{3})\log n/\log ((1-c)\lambda_1)-\eta$ and $r'=\frac{2\epsilon}{3}\cdot \log n/\log ((1-c)\lambda_1)$

\phantom{xxx} Compute $N_{r''[G\backslash E]}(v[i])$ for each $r''\le r+\eta$ and $0\le i<m$.

\phantom{xxx} Run $\text{vertex\_comparison\_algorithm}(v[i],v[j],r,r',E,x)$ for every $i$ and $j$

\phantom{xxx} If these give results consistent with some community memberships which indicate that there is at least one vertex in each community in $v[]$, randomly select one alleged member of each community $v'[\sigma]$. Otherwise, fail.

\phantom{xxx} For every $v''$ in the graph, compute $N_{r''[G\backslash E]}(v'')$ for each $r''\le r'$. Then, run\newline $\text{Vertex\_classification\_algorithm}(v'[],v'', r,r',E,x)$ in order to get a hypothesized classification of $v''$

\phantom{xxx} Return the resulting classification.
\end{algorithm}

\begin{lemma}
For $\epsilon<1$ this algorithm runs in $O(m^2 n^{1-\frac{\epsilon}{3}}+n^{1+\frac{2}{3}\epsilon})$ average time. Assume that all of the following hold:
\begin{align*}
&(1-c)\lambda_{\eta}^4>4\lambda_1^3\\
&0<x<\frac{\lambda_1k}{\lambda_{\eta}\min p_i}\\
&(2(1-c)\lambda_1^3/\lambda_{\eta}^2)^{1-\epsilon/3}<(1-c)\lambda_1\\
&(1+\epsilon/3)>\log((1-c)\lambda_1)/\log ((1-c)\lambda_{\eta}^2/\lambda_1)\\
&13(2x(\min p_j)^{-1/2}+x^2)<\min_{\ne0} (w_i(\{v\})- w_i(\{v'\}))\cdot P^{-1}(w_i(\{v\})- w_i(\{v'\}))
\end{align*}

Let $y=2ke^{-\frac{x^2(1-c)\lambda_{\eta}^2\min p_i}{16\lambda_1k^{3/2}((\min p_i)^{-1/2}+x)}}/\left(1-e^{-\frac{x^2(1-c)\lambda_{\eta}^2\min p_i}{16\lambda_1k^{3/2}((\min p_i)^{-1/2}+x)}\cdot((\frac{(1-c)\lambda_{\eta}^4}{4\lambda_1^3})-1)}\right)$

With probability $1-o(1)$, $G$ is such that $\text{Unreliable\_graph\_classification\_algorithm}(G,c,m,\epsilon,x)$ has at least a \[1-k(1-\min p_i)^m-my\] chance of classifying at least $1-y$ of $G$'s vertices correctly.

\end{lemma}

\begin{proof}
Generating $E$ and $v[]$ takes $O(n)$ time. Computing $N_{r''[G\backslash E]}(v[i])$ for all $r''\le r+\eta$ takes $O(m|\cup_{r''} N_{r''[G\backslash E]}(v[i]))=O(mn)$ time, and computing $N_{r''[G\backslash E]}(v')$ for all $r''\le r'$ and $v'\in G$ takes $$O(n|\cup_{r''\le r'} N_{r''[G\backslash E]})=O(n\cdot ((1-c)\lambda_1)^{r'})=O(n^{1+\frac{2}{3}\epsilon})$$ time. Once these have been computed, running $\text{Vertex\_comparison\_algorithm}(v[i],v[j],r,r',E,x)$ for every $i$ and $j$ takes $O(m^2\cdot ((1-c)\lambda_1)^r)= O(m^2 n^{1-\frac{\epsilon}{3}})$ time, at which point an alleged member of each community can be found in $O(m^2)$ time.  Running $\text{Vertex\_classification\_algorithm}$ on $(v'[],v'', r,r',E,x)$ for every $v''\in G$ takes $O(n\cdot ((1-c)\lambda_1)^{r'})=O(n^{1+\frac{2}{3}\epsilon})$ time. So, the overall algorithm runs in $O(m^2n^{1-\frac{\epsilon}{3}}+n^{1+\frac{2}{3}\epsilon})$ average time.

There exists $y'<y$ such that if these conditions hold, then with probability $1-o(1)$, at least $1-y'$ of $G$'s vertices are $(r+\eta,x)$-good and the number of vertices in $G$ in community $\sigma$ is within $\sqrt{n}\log n$ of $p_\sigma n$ for all $\sigma$. If this is the case, then for sufficiently large $n$, it is at least $1-k(1-\min p_i)^m-my$ likely that every one of the $m$ randomly selected vertices is $(r+\eta,x)$-good and at least one is selected from each community. If that happens, then with probability $1-o(1)$, $\text{vertex\_comparison\_algorithm}(v[i],v[j],r,r',E,x)$ determines whether or not $v[i]$ and $v[j]$ are in the same community correctly for every $i$ and $j$, allowing the algorithm to pick one member of each community. If that happens, then the algorithm will classify each $(r'+\eta,x)$-good vertex correctly with probability $1-o(1)$. So, as long as the initial selection of $v[]$ is good, the algorithm classifies at least $1-y$ of the graph's vertices correctly with probability $1-o(1)$.
\end{proof}

So, this algorithm can sometimes give a vertex classification that is nontrivially better than that obtained by guessing but it has an assymptotically nonzero failure rate. In order to get around that, we combine the results of multiple executions of the algorithm as follows.
\begin{algorithm}
Reliable\_graph\_classification\_algorithm(G,c,m,$\epsilon$,x,T(n)) (i.e., {\tt Sphere-comparison}):

\phantom{xxx} Run $\text{Unreliable\_graph\_classification\_algorithm}(G,c,m,\epsilon,x)$ $T(n)$ times and record the resulting classifications.

\phantom{xxx} Discard any classification that has greater than $$4ke^{-\frac{x^2(1-c)\lambda_{\eta}^2\min p_i}{16\lambda_1k^{3/2}((\min p_i)^{-1/2}+x)}}/\left(1-e^{-\frac{x^2(1-c)\lambda_{\eta}^2\min p_i}{16\lambda_1k^{3/2}((\min p_i)^{-1/2}+x)}\cdot((\frac{(1-c)\lambda_{\eta}^4}{4\lambda_1^3})-1)}\right)$$  disagreement with more than half of the other classifications. In this step, define the disagreement between two classifications as the minimum disagreement over all bijections between their communities.

\phantom{xxx} For every vertex in $G$, randomly pick one of the remaining classifications and assert that it is in the community claimed by that classification, where a community from one classification is assumed to correspond to the community it has the greatest overlap with in each other classification.

\phantom{xxx} Return the resulting combined classification.
\end{algorithm}

\begin{lemma}
For $\epsilon<1$ this algorithm runs in $O(m^2 n^{1-\frac{\epsilon}{3}}T(n)+n^{1+\frac{2}{3}\epsilon}T(n)+nT^2(n))$ average time. Let $y=2ke^{-\frac{x^2(1-c)\lambda_{\eta}^2\min p_i}{16\lambda_1k^{3/2}((\min p_i)^{-1/2}+x)}}/\left(1-e^{-\frac{x^2(1-c)\lambda_{\eta}^2\min p_i}{16\lambda_1k^{3/2}((\min p_i)^{-1/2}+x)}\cdot((\frac{(1-c)\lambda_{\eta}^4}{4\lambda_1^3})-1)}\right)$, and assume that all of the following hold:
\begin{align*}
&(1-c)\lambda_{\eta}^4>4\lambda_1^3\\
&0<x<\frac{\lambda_1k}{\lambda_{\eta}\min p_i}\\
&(2(1-c)\lambda_1^3/\lambda_{\eta}^2)^{1-\epsilon/3}<(1-c)\lambda_1\\
&(1+\epsilon/3)>\log((1-c)\lambda_1)/\log ((1-c)\lambda_{\eta}^2/\lambda_1)\\
&k(1-\min p_i)^m+my<\frac{1}{2}\\
&T(n)=\omega(1)\\
&\min p_i>6y\\
&13(2x(\min p_j)^{-1/2}+x^2)<\min_{\ne0} (w_i(\{v\})- w_i(\{v'\}))\cdot P^{-1}(w_i(\{v\})- w_i(\{v'\}))
\end{align*}

$\text{Reliable\_graph\_classification\_algorithm}(G,c,m,\epsilon,x,T(n))$ classifies as least $$1-2y$$ of $G$'s vertices correctly with probability $1-o(1)$.
\end{lemma}

\begin{proof}
It takes $O(m^2 n^{1-\frac{\epsilon}{3}}T(n)+n^{1+\frac{2}{3}\epsilon}T(n))$ time to run $$\text{Unreliable\_graph\_classification\_algorithm}(G,c,m,\epsilon,x)$$ $T(n)$ times. It takes $O(n)$ time to determine the best bijection between two classification's communities and compute their disagreement. So, it takes $O(nT^2(n))$ time to compute all of the disagreements. Then, it takes $O(n)$ time to combine them and output the result. Therefore, this algorithm takes $O(m^2 n^{1-\frac{\epsilon}{3}}T(n)+n^{1+\frac{2}{3}\epsilon}T(n)+nT^2(n))$ average time.

Assuming the conditions are met, $G$ is such that Unreliable\_graph\_classification\_algorithm on $(G,c,m,\epsilon,x)$ has at least a \[1-k(1-\min p_i)^m-my>\frac{1}{2}\] chance of giving a good classification each time it is run with probability $1-o(1)$. Since $T(n)=\omega(1)$, the majority of the classifications it generates will be good with probability $1-o(1)$. Each good classification has error at most $y$, so any classification with error greater than $3y$ will have disagreement greater than $2y$ with every good classification. On the flip side, no two good classifications can have disagreement greater than $2y$. So, if the majority of the classifications are good, none of the good classifications will be discarded, and any classification with error greater than $3y$ will be discarded.  The requirement that $\min p_i>6y$ ensures that the bijection between any two of the remaining classifications' communities that minimizes their disagreement is the correct bijection. So, classifying each vertex according to one of the remaining bijections chosen at random has a misclassification rate less than $2y$. Therefore, with probability $1-o(1)$, this algorithm classifies at least $1-2y$ of the vertices correctly, as desired.
\end{proof}

\begin{proof}[Proof of Theorem \ref{thm1}]
If the conditions hold, then for all sufficiently small $c$, $\phantom{xxxxxxxxxxxxxxx}$ $\text{Reliable\_graph\_classification\_algorithm}(G,c,\ln(4k)/\min p_i,\epsilon,x,\log n)$ classifies at least  $$1-4ke^{-\frac{(1-c)x^2\lambda_{\eta}^2\min p_i}{16\lambda_1k^{3/2}((\min p_i)^{-1/2}+x)}}/(1-e^{-\frac{(1-c)x^2\lambda_{\eta}^2\min p_i}{16\lambda_1k^{3/2}((\min p_i)^{-1/2}+x)}\cdot((\frac{(1-c)\lambda_{\eta}^4}{4\lambda_1^3})-1)})$$ of $G$'s vertices corrrectly with probability $1-o(1)$. Furthermore, it runs in $O(n^{1+\frac{2}{3}\epsilon}\log n)$ time. Thus, we can get the accuracy arbitrarily close to $$1-4ke^{-\frac{x^2\lambda_{\eta}^2\min p_i}{16\lambda_1k^{3/2}((\min p_i)^{-1/2}+x)}}/(1-e^{-\frac{x^2\lambda_{\eta}^2\min p_i}{16\lambda_1k^{3/2}((\min p_i)^{-1/2}+x)}\cdot((\frac{\lambda_{\eta}^4}{4\lambda_1^3})-1)}),$$ as desired.
\end{proof}



\section{Exact recovery}\label{exact-sec}

Recall that $p$ is a probability vector of dimension $k$, $Q$ is a $k \times k$ symmetric matrix with positive entries, and 
$\gs(n,p,Q)$ denotes the stochastic block model with community prior $p$ and connectivity matrix $\ln(n)Q/n$. A random graph $G$ drawn under $\gs(n,p,Q)$ has a planted community assignment, which we denote by $\sigma \in [k]^n$ and call sometime the true community assignment.

Recall also that exact recovery is solvable for a community partition $[k] = \sqcup_{s=1}^t A_s$, if there exists an algorithm that assigns to each node in $G$ an element of $\{A_1,\dots,A_t\}$ that contains its true community\footnote{Up to a relabelling of the communities.} with probability $1-o_n(1)$.  Exact recovery is solvable in $SBM(n,p,Q)$ if it is solvable for the partition of $[k]$ into $k$ singletons, i.e., all communities can be recovered. 

\subsection{Formal results}
\begin{definition}
Let $\mu,\nu$ be two positive measures on a discrete set $\X$, i.e., two functions from $\mathcal{X}$ to $\mR_+$. We define the CH-divergence between $\mu$ and $\nu$ by 
\begin{align} 
\dd(\mu, \nu) :=\max_{t \in [0,1]} \sum_{x \in \X} \left( t \mu(x) + (1-t)\nu(x)- \mu(x)^t \nu(x)^{1-t} \right). \label{h-div}
\end{align}
\end{definition}
Note that for a fixed $t$, $$\sum_{x \in \X} \left( t \mu(x) + (1-t)\nu(x)- \mu(x)^t \nu(x)^{1-t} \right)$$ is an $f$-divergence. 
For $t=1/2$, i.e., the gap between the arithmetic and geometric means, we have
\begin{align}
\sum_{x \in \X} t\mu(x) + (1-t)\nu(x)- \mu(x)^t \nu(x)^{1-t} = \frac{1}{2} \| \sqrt{\mu}- \sqrt{\nu} \|_2^2
\end{align}
which is the Hellinger divergence (or distance), and the maximization over $t$ of the part $\sum_x \mu(x)^t \nu(x)^{1-t}$ is the exponential of to the Chernoff divergence. 
We refer to Section 8.3 for further discussions on $\dd$. Note also that we will often evaluate $\dd$ as $\dd(x,y)$ where $x,y$ are vectors instead of measures.

\begin{definition}
For the SBM $\gs(n,p,Q)$, where $p$ has dimension $k$ (i.e., there are $k$ communities), the finest partition of $[k]$ is the partition of $[k]$ in to the largest number of subsets such that $\dd((PQ)_i,(PQ)_j) \geq 1$ for all $i,j$ that are in different subsets.
\end{definition}

We next present our main theorem for exact recovery. We first provide necessary and sufficient conditions for exact recovery of partitions, and then provide an algorithm that solves exact recovery efficiently, more precisely, in quasi-linear time. 

\begin{theorem}\label{thm2}
Let $k \in \mZ_+$ denote the number of communities, $p \in (0,1)^k$ with $|p|=1$ denote the community prior, $P=\diag(p)$, and let $Q \in (0,\infty)^{k \times k}$ symmetric with no two rows equal.
\begin{itemize}
\item Exact recovery is solvable in the stochastic block model $\gs(n,p,Q)$ for a partition $[k] = \sqcup_{s=1}^t A_s$ if and only if for all $i$ and $j$ in different subsets of the partition,
\begin{align}
\dd ((PQ)_i , (PQ)_j) \geq 1, \label{d1}
\end{align}
where $(PQ)_i$ denotes the $i$-th row of the matrix $PQ$. In particular, exact recovery is solvable in $\gs(n,p,Q)$ if and only if $\min_{i,j \in [k], i \neq j} \dd ((PQ)_i , (PQ)_j) \geq 1$.
\item For $G \sim \gs(n,p,Q)$, the algorithm {\tt Degree-profiling}$(G,p,Q,\gamma)$ (see below\footnote{$\gamma=\gamma(n,p,Q)$ is set to $\frac{\Delta -1}{2\Delta} +\frac{\ln\ln n}{4 \ln n}$, where $\Delta = \min_{r,s \in [t] \atop{r \neq s}} \min_{ i \in A_r, j \in A_s}  \dd((pQ)_i, (pQ)_j)$ and $A_1,\dots,A_t$ is the finest partition of $[k]$.}) recovers the finest partition with probability $1-o_n(1)$ and runs in $o(n^{1+\epsilon})$ time for all $\epsilon>0$. 
\end{itemize} 
\end{theorem}
Note that second item in the theorem implies that {\tt Degree-profiling} solves exact recovery efficiently whenever the parameters $p$ and $Q$ allow for exact recovery to be solvable. In addition, it gives an operational meaning to a new divergence function, analog to operational meaning given to the KL-divergence in the channel coding theorem (see Section \ref{it-inter}).

\begin{remark}\label{qzero}
f $Q_{ij}=0$ for some $i$ and $j$ then the results above still hold, except that if for all $i$ and $j$ in different subsets of the partition,
\begin{align}
\dd ((PQ)_i , (PQ)_j) \geq 1, \label{d1}
\end{align}
but there exist $i$ and $j$  in different subsets of the partition such that $\dd ((PQ)_i , (PQ)_j) = 1$ and $((PQ)_{i,k}\cdot (PQ)_{j,k}\cdot ((PQ)_{i,k}-(PQ)_{j,k})=0$ for all $k$, then the optimal algorithm will have an assymptotically constant failure rate. The recovery algorithm also needs to be modified to accomodate $0$'s in $Q$.
\end{remark}

\begin{remark}
As shown in the proof of Theorem \ref{thm2}, when exact recovery is not solvable, any algorithm must confuse at least one vertex with probability $1-o_n(1)$, and not just with probability away from 0. Hence exact recovery has a sharp threshold at $\min_{i,j \in [k], i \neq j} \dd ((PQ)_i || (PQ)_j) = 1$.
\end{remark}

\begin{example}
For the symmetric block model where $p$ is equiprobable on $[k]$ and $Q$ takes only two different values, $\alpha$ on the diagonal and $\beta$ outside the diagonal (with $\alpha,\beta>0$), the requirement in Theorem \ref{thm2} for recovery of any (or all) communities is equivalent to  
\begin{align}
|\sqrt{\alpha} - \sqrt{\beta}| \geq \sqrt{k},
\end{align}
which generalizes the result obtained in \cite{abh,mossel-consist} for $k=2$.  
\end{example}

The algorithm {\tt Degree-profiling} is given in Section \ref{pt1} and replicated below. The idea is to recover the communities with a two-step procedure, similarly to one of the algorithms used in \cite{abh} for the two-community case. In the first step, we run {\tt Sphere-comparison} on a sparsified version of $\gs(n,p,Q)$ which has a slowly growing average degree. Hence, from Corollary \ref{partial-delta}, {\tt Sphere-comparison} recovers correctly a fraction of nodes that is arbitrarily close to 1 (w.h.p.). In the second step, we proceed to an improvement of the first step classification by making local checks for each node in the residue graph and deciding whether the node should be moved to another community or not. This step requires solving a hypothesis testing problem for deciding the local degree profile of vertices in the SBM. The CH-divergence appears when resolving this problem, as the mis-classification error exponent. We present this result of self-interest in Section \ref{testing}. The proof of Theorem \ref{thm2} is given in Section \ref{proof-thm-2}. \\

{\tt Degree-profiling} algorithm. \\
Inputs: a graph $g=([n],E)$, the SBM parameters $p_i$, $i \in [k]$, $Q_{i,j}$, $i,j \in [k]$, and a splitting parameter $\gamma \in [0,1]$ (see Theorem \ref{thm2} for the choice of $\gamma$).\\
Output: Each node $v \in [n]$ is assigned a community-list $A(v) \in \{A_1,\dots,A_t\}$, where $A_1,\dots,A_t$ is the partition of $[k]$ in to the largest number of subsets such that $\dd((pQ)_i,(pQ)_j) \geq 1$ for all $i,j$ in $[k]$ that are in different subsets.\\
Algorithm:\\
(1) Define the graph $g'$ on the vertex set $[n]$ by selecting each edge in $g$ independently with probability $\gamma$, and define the graph $g''$ that contains the edges in $g$ that are not in $g'$. \\
(2) Run {\tt Sphere-comparison} on $g'$ to obtain the preliminary classification $\sigma' \in [k]^n$ (see Section \ref{partial-sec} and Corollary \ref{partial-delta}.) \\
(3) Determine the edge density between each pair of alleged communities, and use this information and the alleged communities' sizes to attempt to identify the communities up to symmetry.\\
(4) For each node $v \in [n]$, determine in which community node $v$ is most likely to belong to based on its degree profile computed from the preliminary classification $\sigma'$ (see Section \ref{testing}), and call it $\sigma''_v$\\
(5) For each node $v \in [n]$, determine in which group $A_1,\dots,A_t$ node $v$ is most likely to belong to based on its degree profile computed from the preliminary classification $\sigma''$ (see Section \ref{testing}). 

\subsection{Testing degree profiles}\label{testing}
In this section, we consider the problem of deciding which community a node in the SBM belongs to based on its degree profile. We first make the latter terminology precise. 
\begin{definition}
The degree profile of a node $v \in [n]$ for a partition of the graph's vertices into $k$ communities is the vector $d(v) \in \mZ_+^k$, where the $j$-th component $d_j(v)$ counts the number of edges between $v$ and the vertices in community $j$.  Note that $d(v)$ is equal to $N_1(v)$ as defined in Definition \ref{def-n1}.
\end{definition}

For $G \sim \gs(n,p,Q)$, community $i \in [k]$ has a relative size that concentrates exponentially fast to $p_i$. Hence, for a node $v$ in community $j$, $d(v)$ is approximately given by $\sum_{i \in [k]} X_{ij}e_i$, 
where $X_{ij}$ are independent and distributed as $\bin(np_i,\ln(n)Q_{i,j}/n)$, and where $\bin(a,b)$ denotes\footnote{$\bin(a,b)$ refers to $\bin( \lfloor a \rfloor,b)$ if $a$ is not an integer.} the binomial distribution with $a$ trials and success probability $b$. Moreover, the Binomial is well-enough approximated by a Poisson distribution of the same mean in this regime. In particular, Le Cam's inequality gives 
\begin{align}
\left\| \bin \left(n a, \frac{\ln(n)}{n} b \right) -  \mathcal{P}\left(a b  \ln(n) \right) \right\|_{TV} \leq 2 \frac{a b^2 \ln^2(n)}{n},
\end{align}
hence, by the additivity of Poisson distribution and the triangular inequality, 
\begin{align}
\| \mu_{d(v)} -  \mathcal{P}(\ln(n) \sum_{i \in [k]} p_i Q_{i,j}e_i) \|_{TV} = O \left(\frac{\ln^2(n)}{n} \right).
\end{align} 
We will rely on a simple one-sided bound (see \eqref{bipo}) to approximate our events under the Poisson measure. 

Consider now the following problem. Let $G$ be drawn under the $\gs(n,p,Q)$ SBM and assume that the planted partition is revealed except for a given vertex. Based on the degree profile of that vertex, is it possible to classify the vertex correctly with high probability? We have to resolve a hypothesis testing problem, which involves multivariate Poisson distributions in view of the previous observations. We next study this problem.\\ 

{\bf Testing multivariate Poisson distributions.} Consider the following Bayesian hypothesis testing problem with $k$ hypotheses.
The random variable $H$ takes values in $[k]$ with $\pp\{H=j\}=p_j$ (this is the a priori distribution of $H$). Under $H=j$, an observed random variable $D$ is drawn from a multivariate Poisson distribution with mean $\lambda(j) \in \mR_+^k$, i.e.,  
\begin{align}
\pp\{D=d|H=j\}=\mathcal{P}_{\lambda(j)}(d), \quad d \in \mZ_+^k,
\end{align}
where 
\begin{align}
\mathcal{P}_{\lambda(j)}(d) = \prod_{i \in [k]} \mathcal{P}_{\lambda_i(j)}(d_i), 
\end{align}
and
\begin{align}
\mathcal{P}_{\lambda_i(j)}(d_i) = \frac{\lambda_i(j)^{d_i}}{d_i!} e^{- \lambda_i(j)}.
\end{align} 
In other words, $D$ has independent Poisson entries with different means. We use the following notation to summarize the above setting:
\begin{align}
D|H=j \,\, \sim \mathcal{P}(\lambda(j)), \quad j \in [k].
\end{align}
Our goal is to infer the value of $H$ by observing a realization of $D$. To minimize the error probability given a realization of $D$, we must pick the most likely hypothesis conditioned on this realization, i.e., 
\begin{align}
\argmax_{j \in [k]} \pp \{D=d | H=j\} p_j, \label{map-rule}
\end{align}
which is the Maximum A Posteriori (MAP) decoding rule.\footnote{Ties can be broken arbitrarily.} To resolve this maximization, we can proceed to a tournament of $k-1$ pairwise comparisons of the hypotheses. Each comparison allows us to eliminate one candidate for the maxima, i.e., 
\begin{align}
\pp \{D=d | H=i\} p_i >  \pp \{D=d | H=j\} p_j \quad \Rightarrow \quad H \neq  j. 
\end{align}
The error probability $P_e$ of this decoding rule is then given by,
\begin{align}
P_e = \sum_{i \in [k]} \pp\{ D \in \text{Bad}(i)  | H=i\} p_i, \label{error1}
\end{align}
where $\text{Bad}(i)$ is the region in $\mZ_+^k$ where $i$ is not maximizing \eqref{map-rule}. Moreover, for any $i \in [k]$, 
\begin{align}
\pp\{ D \in \text{Bad}(i) | H=i\} \leq \sum_{j \neq i} \pp\{ D \in \text{Bad}_j(i) |H=i\} \label{pair1}
\end{align}
where $\text{Bad}_j(i)$ is the region in $\mZ_+^k$ where $\pp \{D=x | H=i\} p_{i} \leq  \pp \{D=x | H=j\} p_{j}$. 
Note that with this upper-bound, we are counting the overlap regions where $\pp \{D=x | H=i\} p_{i} \leq  \pp \{D=x | H=j\} p_{j}$ for different $j$'s multiple times, but no more than $k-1$ times. Hence,  
\begin{align}
\sum_{j \neq i} \pp\{ D \in \text{Bad}_j(i) |H=i\}  \leq (k-1) \pp\{ D \in \text{Bad}(i) | H=i\}. \label{dbound1}
\end{align}
Putting \eqref{error1} and \eqref{pair1} together, we have 
\begin{align}
P_e &\leq  \sum_{i \neq j} \pp\{ D \in \text{Bad}_j(i)  | H=i\}   p_i,\\
 &=\sum_{i < j} \sum_{d \in \mZ_+^k} \min(\pp \{D=d | H=i\} p_{i} ,  \pp \{D=d | H=j\} p_{j}) \label{ub1}
\end{align}
and from \eqref{dbound1},
\begin{align}
P_e &\geq \frac{1}{k-1} \sum_{i < j} \sum_{d \in \mZ_+^k} \min(\pp \{D=d | H=i\} p_{i} ,  \pp \{D=d | H=j\} p_{j}). \label{ub2}
\end{align}
Therefore the error probability $P_e$ can be controlled by estimating the terms $\sum_{d \in \mZ_+^k} \min(\pp \{D=d | H=i\} p_{i} ,  \pp \{D=d | H=j\} p_{j})$. In our case, recall that  
\begin{align}
\pp \{D=d | H=i\} = \mathcal{P}_{\lambda(i)}(d),
\end{align}
which is a multivariate Poisson distribution. In particular, we are interested in the regime where $k$ is constant and $\lambda(i) =  \ln(n) c_i$, $c_i \in \mR_+^k$, and $n$ diverges. 
Due to \eqref{ub1}, \eqref{ub2}, we can then control the error probability by controlling $\sum_{x \in \mZ_+^k} \min(\mathcal{P}_{\ln(n) c_i}(x) p_{i} ,  \mathcal{P}_{\ln(n) c_j}(x) p_{j})$, which we will want to be $o(1/n)$ to classify vertices in the SBM correctly with high probability based on their degree profiles (see next section). The following lemma provides the relevant estimates.  

\begin{lemma}\label{hell-expo}
For any $c_1, c_2 \in (\mR_+\setminus \{0\})^k$ with $c_1 \neq c_2$ and $p_1,p_2 \in \mR_+\setminus \{0\}$, 
\begin{align}
& \sum_{x \in \mZ_+^k} \min(\mathcal{P}_{\ln(n) c_1}(x) p_{1} ,  \mathcal{P}_{\ln(n) c_2}(x) p_{2}) = O\left(n^{- \dd(c_1,c_2) - \frac{\ln\ln(n)}{2 \ln(n)}} \right),\\
& \sum_{x \in \mZ_+^k} \min(\mathcal{P}_{\ln(n) c_1}(x) p_{1} ,  \mathcal{P}_{\ln(n) c_2}(x) p_{2}) = \Omega \left(n^{- \dd(c_1,c_2) - \frac{k \ln\ln(n)}{2 \ln(n)}} \right),
\end{align}
where $\dd(c_1,c_2)$ is the CH-divergence as defined in \eqref{h-div}.
\end{lemma}
In other words, the CH-divergence provides the error exponent for deciding among multivariate Poisson distributions. We did not find this result in the literature, but found a similar result obtained by Verd\'u \cite{verdu-hell}, who shows that the Hellinger distance (the special case with $t=1/2$ instead of the maximization over $t$) appears in the error exponent for testing Poisson point-processes, although \cite{verdu-hell} does not investigate the exact error exponent.  

\begin{proof}[Proof of Lemma \ref{hell-expo}]
Assume without loss of generality that $c_{1,1}\ne c_{2,1}$. To prove the first half of the lemma, note that for any $t\in [0,1]$,
\begin{align}
& \sum_{x \in \mZ_+^k} \min(\mathcal{P}_{\ln(n) c_1}(x) p_{1} ,  \mathcal{P}_{\ln(n) c_2}(x) p_{2})\\
&\le \max(p_1,p_2)\sum_{x \in \mZ_+^k} \min(\mathcal{P}_{\ln(n) c_1}(x),  \mathcal{P}_{\ln(n) c_2}(x))\\
&=\max(p_1,p_2)\sum \min(e^{-\ln n\sum c_{1,i}}\prod (\ln n\cdot c_{1,i})^{x_i}/x_i!,e^{-\ln n\sum c_{2,i}}\prod (\ln n\cdot c_{2,i})^{x_i}/x_i!)\\
&=\max(p_1,p_2) e^{-\ln n\sum tc_{1,i}+ (1-t)c_{2,i} }\sum \left(\ln n\cdot c_{1,i}^t c_{2,i}^{1-t}\right)^{x_i}/x_i!\\
&\phantom{xxxxxx}\cdot \min\left(e^{-\ln n(1-t)\sum  c_{1,i}- c_{2,i}}\prod( c_{1,i}/ c_{2,i})^{(1-t)x_i}, e^{-\ln n t\sum  c_{2,i}- c_{1,i}}\prod( c_{2,i}/ c_{1,i})^{tx_i}\right)
\end{align}

For any choice of $x_2,...,x_k$, there must exist $x_1$ (not necessarily an integer) such that $e^{-\ln n(1-t)\sum  c_{1,i}- c_{2,i}}\prod( c_{1,i}/ c_{2,i})^{(1-t)x_i}=1$. As a result, the expression above must be less than or equal to

\begin{align}
&\max(p_1,p_2) e^{-\ln n\sum tc_{1,i}+ (1-t)c_{2,i} }\sum \left(\ln n\cdot c_{1,i}^t c_{2,i}^{1-t}\right)^{x_i}/x_i!\\
&\cdot \min\left ((c_{1,1}/c_{2,1})^{1-t},(c_{2,1}/c_{1,1})^t\right)^{|x_1-\ln n c_{1,1}^tc_{2,1}^{1-t}|-1}\\
&=O\left(e^{\ln n\sum (c_{1,i}^t c_{2,i}^{1-t}- tc_{1,i}- (1-t)c_{2,i})}/\sqrt{\ln n}\right)
\end{align} 
When $t$ is chosen to maximize $\sum tc_{1,i}+(1-t)c_{2,i}-c_{1,i}^tc_{2,i}^{1-t}$, this is $$O\left(n^{- \dd(c_1,c_2) - \frac{\ln\ln(n)}{2 \ln(n)}} \right).$$
To prove the second half, let $t$ maximize $\sum tc_{1,i}+(1-t)c_{2,i}-c_{1,i}^tc_{2,i}^{1-t}$. Hence $$\sum c_{1,i}-c_{2,i}-\ln (c_{1,i}/c_{2,i})c_{1,i}^tc_{2,i}^{1-t}=0.$$ This implies that 
\[e^{-\ln n\sum c_{1,i}}\prod (\ln n\cdot c_{1,i})^{\ln nc_{1,i}^tc_{2,i}^{1-t}}=e^{-\ln n\sum c_{2,i}}\prod (\ln n\cdot c_{2,i})^{\ln nc_{1,i}^tc_{2,i}^{1-t}}.\] As a result, 
\begin{align}
&\min(\mathcal{P}_{\ln(n) c_1}(c_{1,i}^tc_{2,i}^{1-t}\ln n) p_{1} ,  \mathcal{P}_{\ln(n) c_2}c_{1,i}^tc_{2,i}^{1-t}\ln n) p_{2})\\
&\ge\min(p_1,p_2) e^{-\ln n\sum c_{1,i}}\prod (\ln n\cdot c_{1,i})^{c_{1,i}^tc_{2,i}^{1-t}\ln n}/(c_{1,i}^tc_{2,i}^{1-t}\ln n)!\\
&=\min(p_1,p_2)e^{-\ln n\sum c_{2,i}}\prod (\ln n\cdot c_{2,i})^{c_{1,i}^tc_{2,i}^{1-t}\ln n}/(c_{1,i}^tc_{2,i}^{1-t}\ln n)!\\
&=\min(p_1,p_2)e^{-\ln n\sum tc_{1,i}+(1-t)c_{2,i}}\prod (\ln n\cdot c_{1,i}^tc_{2,i}^{1-t})^{c_{1,i}^tc_{2,i}^{1-t}\ln n}/(c_{1,i}^tc_{2,i}^{1-t}\ln n)!\\
&=\Omega(\min(p_1,p_2)e^{-\ln n\sum tc_{1,i}+(1-t)c_{2,i}-c_{1,i}^tc_{2,i}^{1-t}}/(\ln n)^{k/2})\\
&=\Omega \left(n^{- \dd(c_1,c_2) - \frac{k \ln\ln(n)}{2 \ln(n)}} \right).
\end{align}
Thus, 
\begin{align}
\sum_{x \in \mZ_+^k} \min(\mathcal{P}_{\ln(n) c_1}(x) p_{1} ,  \mathcal{P}_{\ln(n) c_2}(x) p_{2}) = \Omega \left(n^{- \dd(c_1,c_2) - \frac{k \ln\ln(n)}{2 \ln(n)}} \right) .
\end{align}
\end{proof}


This lemma together with previous bounds on $P_e$ imply that if $\dd(c_i,c_j) > 1$ for all $i \neq j$, the true hypothesis is correctly recovered with probability $o(1/n)$. However, it may be that $\dd(c_i,c_j) > 1$ only for a subset of $(i,j)$-pairs. What can we then infer? While we may not recover the true value of $H$ with probability $o(1/n)$, we may narrow down the search within a subset of possible hypotheses with that probability of error.\\

{\bf Testing composite multivariate Poisson distributions.} We now consider the previous setting, but we are no longer interested in determining the true hypothesis, but in deciding between two (or more) disjoint subsets of hypotheses. Under hypothesis 1, the distribution of $D$ belongs to a set of possible distributions, namely $\mathcal{P}(\lambda_i)$ where $i \in A$, and under hypothesis 2, the distribution of $D$ belongs to another set of distributions, namely $\mathcal{P}(\lambda_i)$ where $i \in B$. Note that $A$ and $B$ are disjoint subsets such that $A \cup B=[k]$. In short,
\begin{align}
D|\tilde{H}=1 \,\, \sim \mathcal{P}(\lambda_i), \,\,\, \text{for some } i \in A, \\
D|\tilde{H}=2 \,\, \sim \mathcal{P}(\lambda_i), \,\,\, \text{for some } i \in B,
\end{align}
and as before the prior on $\lambda_i$ is $p_i$. To minimize the probability of deciding the wrong hypothesis upon observing a realization of $D$, we must pick the hypothesis which leads to the larger probability between $\pp\{\tilde{H} \in A | D=d\}$ and $\pp\{\tilde{H} \in B | D=d\}$, or equivalently, 
\begin{align}
\sum_{i \in A} \mathcal{P}_{\lambda(i)}(d) p_{i} \geq \sum_{i \in B} \mathcal{P}_{\lambda(i)}(d) p_{i} \quad \Rightarrow \quad \tilde{H}=1,\\
\sum_{i \in A} \mathcal{P}_{\lambda(i)}(d) p_{i} < \sum_{i \in B} \mathcal{P}_{\lambda(i)}(d) p_{i} \quad \Rightarrow \quad \tilde{H}=2.
\end{align}
In other words, the problem is similar to previous one, using the above mixture distributions. 
If we denote by $\tilde{P}_e$ the probability of making an error with this test, we have 
\begin{align}
\tilde{P}_e &= \sum_{x \in \mZ_+^k} \min\left(\sum_{i \in A} \mathcal{P}_{\lambda(i)}(x) p_{i} , \sum_{i \in B} \mathcal{P}_{\lambda(i)}(x) p_{i} \right).
\end{align}
Moreover, applying bounds on the minima of two sums,   
\begin{align}
\tilde{P}_e & \leq  \sum_{x \in \mZ_+^k} \sum_{i \in A, j \in B} \min\left(\mathcal{P}_{\lambda(i)}(x) p_{i} ,  \mathcal{P}_{\lambda(j)}(x) p_{j}\right),\\
\tilde{P}_e & \geq  \frac{1}{|A||B|} \sum_{x \in \mZ_+^k} \sum_{i \in A, j \in B} \min\left(\mathcal{P}_{\lambda(i)}(x) p_{i} ,  \mathcal{P}_{\lambda(j)}.(x) p_{j}\right) .
\end{align} 
Therefore, for constant $k$ and $\lambda(i) =  \ln(n) c_i$, $c_i \in \mR_+^k$, with $n$ diverging, it suffices to control the decay of $\sum_{x \in \mZ_+^k}  \min(\mathcal{P}_{\lambda(i)}(x) p_{i} ,  \mathcal{P}_{\lambda(j)}(x) p_{j})$ when $i \in A$ and $j \in B$, in order to bound the error probability of deciding whether a vertex degree profile belongs to a group of communities or not. 

The same reasoning can be applied to the problem of deciding whether a given node belongs to a group of communities, with more than two groups. Also, for any $p$ and $p'$ such that $|p_j-p'_j|<\ln n/\sqrt{n}$ for each $j$, $Q$, $\gamma(n)$, and $i$, 
\begin{align}
\sum_{x\in \mZ_+^k}\max\left(Bin_{\left(np',\frac{(1-\gamma(n))\ln(n)}{n}Q_i\right)}(x)- 2\mathcal{P}_{PQ_i(1-\gamma(n))\ln(n)/n}(x),0\right)=O(1/n^2). \label{bipo} 
\end{align}
So, the error rate for any algorithm that classifies vertices based on their degree profile in a graph drawn from a sparse SBM is at most $O(1/n^2)$ more than twice what it would be if the probability distribution of degree profiles really was the poisson distribution.

In summary, we have proved the following. 
\begin{lemma}\label{testing-lemma}
Let $k \in \mZ_+$ and let $A_1,\dots,A_t$ be disjoint subsets of $[k]$ such that $\cup_{i=1}^t A_i = [k]$. Let $G$ be a random graph drawn under $\gs(n,p,(1-\gamma(n))Q)$. Assigning the most likely community subset $A_i$ to a node $v$ based on its degree profile $d(v)$ gives the correct assignment with probability $$1-O \left(n^{-(1-\gamma(n))\Delta -\frac{1}{2}\ln((1-\gamma(n))\ln n)/\ln n } +\frac{1}{n^2}\right),$$ 
where
\begin{align}
\Delta = \min_{r,s \in [t] \atop{r \neq s}} \min_{ i \in A_r, j \in A_s}  \dd((pQ)_i, (pQ)_j).
\end{align}
\end{lemma}
Moreover, we will need the following ``robust'' version of this lemma to prove Theorem \ref{thm2}. 
\begin{lemma}\label{testing-lemma2}
Let $k \in \mZ_+$ and let $A_1,\dots,A_t$ be disjoint subsets of $[k]$ such that $\cup_{i=1}^t A_i = [k]$. Let $G$ be a random graph drawn under $\gs(n,p,(1-\gamma(n))Q)$. There exist $c_1$, $c_2$, and $c_3$ such that for any $\delta$, assigning the most likely community subset $A_i$ to a node $v$ based on a distortion of its degree profile that independently gets each node's community wrong with probability at most $\delta$ gives the correct assignment with probability at least $$1- c_2\cdot (1+c_1\delta)^{c_3\ln n} \cdot \left(n^{-(1-\gamma(n))\Delta -\frac{1}{2}\ln((1-\gamma(n))\ln n)/\ln n) } \right)-\frac{1}{n^2},$$ 
where
\begin{align}
\Delta = \min_{r,s \in [t] \atop{r \neq s}} \min_{ i \in A_r, j \in A_s}  \dd((pQ)_i, (pQ)_j).
\end{align}
\end{lemma}

\begin{proof}
Let $$c_1=\max_{i,j} \sum p_{i'}q_{i',j}/(p_i q_{i,j}).$$ The key observation is that $v$'s $m$th neighbor had at least a $\min_{i,j} (p_i q_{i,j})/\sum p_{i'}q_{i',j}$ chance of actually being in community $\sigma$ for each $\sigma$, so its probability of being reported as being in community $\sigma$ is at most $1+c_1\delta$ times the probability that it actually is. So, the probability that its reported degree profile is bad is at most $(1+c_1\delta)^{|N_1(v)|}$ times the probability that its actual degree profile is bad. Choose $c_3$ such that each vertex in the graph has degree less than $c_3\ln n$ with probability $1-\frac{1}{n^2}$ and the conclusion follows from this and the previous bounds on the probability that classifying a vertex based on its degree profile fails.
\end{proof}

\subsection{Proof of Theorem \ref{thm2}}\label{proof-thm-2}
We break the proof into two parts, the possibility and impossibility parts.  \\

\noindent
{\bf Claim 1 (achievability).} Let $G \sim \gs(n,p,Q)$ and $\gamma=  \frac{\Delta -1}{2 \Delta} +\frac{\ln\ln n}{4 \ln n}$, where $$\Delta = \min_{r,s \in [t] \atop{r \neq s}} \min_{ i \in A_r, j \in A_s}  \dd((pQ)_i, (pQ)_j)$$ and $A_1,\dots,A_t$ is a partition of $[k]$. {\tt Degree-profiling}$(G,p,Q,\gamma)$ recovers the partition $[k] = \sqcup_{s=1}^t A_s$ with probability $1-o_n(1)$ if for all $i,j$ in $[k]$ that are in different subsets,
\begin{align}
\dd ((PQ)_i , (PQ)_j) \geq 1. \label{d1b}
\end{align}

The idea behind Claim 1 is contained in Lemma \ref{testing-lemma}. However, there are several technical steps that need to be handled:
\begin{enumerate}
\item The graphs $G'$ and $G''$ obtained in step 1 of the algorithm are correlated, since an edge cannot be both in $G'$ and $G''$. However, this effect can be discarded since two independent versions would share edges with low enough probability. 
\item  The classification in step 2 using {\tt Sphere-comparison} has a vanishing fraction of vertices which are wrongly labeled. This requires using the robust version of Lemma \ref{testing-lemma}, namely Lemma \ref{testing-lemma2}. 
\item In the case where $\dd ((PQ)_i , (PQ)_j) = 1$ a more careful classification is needed as carried in steps 3 and 4 of the algorithm.    
\end{enumerate} 

\begin{proof}
With probability $1-O(1/n)$, no vertex in the graph has degree greater than $c_3\ln n$. Assuming that this holds, no vertex's set of neighbors in $G''$ is more than $$(1-\max q_{i,j} \ln n/n)^{-c_3\ln n}\cdot (n/(n-c_3\ln n))^{c_3\ln n}=1+o(1)$$ times as likely to occur as it would be if $G''$ were independent of $G'$. So, the fact that they are not has negligible impact on the algorithm's error rate. As $n$ goes to infinity, $\sigma'$'s expected error rate goes to $0$, so the expected error in the algorithm's estimates of the connectivity rates between its communities also goes to $0$. Thus, the algorithm labels the communities correctly (up to equivalent relabeling) with probability $1-o(1)$. Now, let $$\delta=(e^{\frac{(1-\gamma)}{2c_3} \min_{i\ne j} \dd((PQ)_i,(PQ)_j)}-1)/c_1.$$ By Lemma \ref{testing-lemma2}, if the classification in step $2$ has an error rate of at most $\delta$, then the classification in step $3$ has an error rate of $$O(n^{- (1-\gamma)\min_{i\ne j} \dd((PQ)_i,(PQ)_j)/2}+1/n^2),$$ observing that if $\sigma'_v\ne \sigma_v$ the error rate of $\sigma''_{v'}$ for $v'$ adjacent to $v$ is at worst multiplied by a constant. That in turn ensures that the final classification has an error rate of at most \[O\left((1+O(n^{- (1-\gamma)\min_{i\ne j} \dd((PQ)_i,(PQ)_j)/2}+1/n^2))^{c_3\ln n} \frac{1}{n}\ln{n}^{-1/4}\right)=O\left(\frac{1}{n}\ln n^{-1/4}\right) .\]
\end{proof}

\noindent
{\bf Claim 2 (converse).} Let $G \sim \gs(n,p,Q)$ and $A_1,\dots,A_t$ a partition of $[k]$. If there exist $r,s \in [t]$, $s\neq t$, $i \in A_r$, and $j \in A_s$ such that 
\begin{align}
\dd ((PQ)_i , (PQ)_j) < 1, 
\end{align}
then every algorithm classifying the vertices of $G$ into elements $A_1,\dots,A_t$ must mis-classify at least one vertex with probability $1-o_n(1)$.

\begin{proof}
With probability $1-o(1)$, every community of $G$ has a size that is within a factor of $1+\ln n/\sqrt{n}$ of its expected size. Assume that this holds. Let $S$ be a random set of $n/\ln^3(n)$ of $G$'s vertices. With probability $1-o(1)$ the number of vertices in $S$ in community $\ell$ is within $\sqrt{n}$ of $p_\ell n/\log^3 n$ for each $\ell$, and a randomly selected vertex in $S$ is adjacent to another vertex in $S$ with probability $o(1)$. Next, choose $i$ and $j$ such that $i\in A_r$, $j\in A_s$, $r\ne s$, and $\dd ((PQ)_i , (PQ)_j) < 1$. Now, let 
\begin{align}
x_\ell=\lfloor (PQ)_{\ell,i}^t (PQ)_{\ell,j}^{1-t}\ln n\rfloor
\end{align}
for each $\ell \in [k]$, where $t\in [0,1]$ is chosen to maximize 
\begin{align}
\sum t(PQ)_{\ell,i}+(1-t)(PQ)_{\ell,j}-(PQ)_{\ell,i}^t(PQ)_{\ell,j}^{1-t}.
\end{align}
Roughly speaking, we want to show that every vertex in community $i$ or $j$ has a degree profile of $x$ with probability 
\begin{align}
\Omega(n^{-\dd ((PQ)_i , (PQ)_j)}/\ln^{k/2}(n))
\end{align}
and that there is no way to determine which community the vertices with this degree profile are in.

More precisely, call a vertex in $S$ ambiguous if it has exactly $x_\ell$ neighbors in community $\ell$ that are not in $S$ for each $\ell$. The probability distribution of a vertex's numbers of neighbors in $G\backslash S$ in each community is essentially a multivariable poisson distribution. So, by the assumption that $\dd ((PQ)_i , (PQ)_j) < 1$ and the argument in Lemma $\ref{hell-expo}$, there exists $\epsilon>0$ such that a vertex in $S$ that is in either community $i$ or community $j$ is ambiguous with probability $\Omega(n^{\epsilon-1})$. Furthermore, for a fixed community assignment and choice of $S$, there is no dependence between whether or not any two vertices are ambiguous. Also, an ambiguous vertex is not adjacent to any other vertex in $S$ with probability $1-o(1)$. So, with probability $1-o(1)$ there are at least $\ln n$ ambigous vertices in community $i$ that are not adjacent to any other vertices in $S$ and $\ln n$ ambiguous vertices in community $j$ that are not adjacent to any other vertices in $S$. These vertices are indistinguishable, so no algorithm classifies them all correctly with probability greater than $1/{2\ln n\choose \ln n}$. Therefore, every algorithm must misclassify at least one vertex with probability $1-o(1)$. 
\end{proof}


\section*{Acknowledgements}
We would like to thank Bell Labs for supporting part of this research, as well as Sergio Verd\'u and Imre Csisz\'ar for useful discussions on $f$-divergences. 



\bibliographystyle{amsalpha}
\bibliography{gen2}

\end{document}